\documentclass[a4paper,11pt,reqno]{amsart}
\usepackage[margin=1in]{geometry}
\usepackage{amsmath,amsthm,amssymb,amsfonts,extarrows,mathrsfs,comment,tikz-cd,color,graphicx,hyperref}
\usepackage{url}
\usepackage{stmaryrd}
\usepackage{dsfont}

\usepackage{mathtools}
\usepackage{enumerate}
\usepackage[all]{xy}
\usepackage{indentfirst}
\usepackage{mathtools}
\usepackage{setspace}
\DeclareMathOperator{\Stab}{Stab}

\usepackage{graphicx,epstopdf}
\newif\ifHideFoot
\HideFoottrue  
\HideFootfalse

\ifHideFoot

\newcommand{\Zhiyuan}[1]{}
\newcommand{\Ruxuan}[1]{}

\else 

\newcommand{\marg}[1]{\normalsize{{
			\color{red}\footnote{{\color{blue}#1}}}{\marginpar[\vskip
			-.25cm{\color{red}\hfill$\Rightarrow$\tiny\thefootnote}]{\vskip
				-.2cm{\color{red}$\Leftarrow$\tiny\thefootnote}}}}}

\newcommand{\Zhiyuan}[1]{\marg{(Zhiyuan) #1}}
\newcommand{\Ruxuan}[1]{\marg{(Ruxuan) #1}}

\fi

\newtheorem{theorem}{Theorem}[section]
\newtheorem{proposition}[theorem]{Proposition}
\newtheorem{conjecture}[theorem]{Conjecture}
\newtheorem{corollary}[theorem]{Corollary}
\newtheorem{lemma}[theorem]{Lemma}

\newtheorem{claim}{Claim}

\theoremstyle{definition}

\newtheorem{definition}[theorem]{Definition}

\newtheorem{remark}[theorem]{Remark}

\newtheorem{example}[subsection]{Example}

\newcommand\cC{\mathcal{C}}

\newcommand\cE{\mathcal{E}}
\newcommand\cF{\mathcal{F}}
\newcommand\cG{\mathcal{G}}
\newcommand\cH{\mathcal{H}}

\newcommand\cL{\mathcal{L}}

\newcommand\cO{\mathcal{O}}

\newcommand\cQ{\mathcal{Q}}

\newcommand\cT{\mathcal{T}}

\newcommand\cZ{\mathcal{Z}}

\newcommand\CC{\mathbb{C}}

\newcommand\LL{\mathbb{L}}

\newcommand\NN{\mathbb{N}}

\newcommand\PP{\mathbb{P}}
\newcommand\QQ{\mathbb{Q}}
\newcommand\RR{\mathbb{R}}

\newcommand\ZZ{\mathbb{Z}}
\newcommand\scE{\mathscr{E}}
\newcommand\scG{\mathscr{G}}

\newcommand\cf{\textit{cf}}

\newcommand\bB{\mathbf{B}}

\newcommand\bP{\mathbf{P}}

\newcommand\bR{\mathbf{R}}
\newcommand\bS{\mathbf{S}}

\newcommand\rD{\mathrm{D}}

\newcommand\rH{\mathrm{H}}

\newcommand\rR{\mathrm{R}}

\newcommand\bfv{\mathbf{v}}
\newcommand\bfw{\mathbf{w}}

\newcommand\frh{\mathfrak{h}}

\newcommand\fro{\mathfrak{o}}

\newcommand\rd{\mathrm{d}}

\DeclareMathOperator{\CH}{CH}

\DeclareMathOperator{\codim}{codim}
\DeclareMathOperator{\rank}{rank}

\DeclareMathOperator{\Br}{Br}
\DeclareMathOperator{\fl}{fl}
\DeclareMathOperator{\id}{id}
\DeclareMathOperator{\Pic}{Pic}
\DeclareMathOperator{\NS}{NS}

\DeclareMathOperator{\Hom}{Hom}
\DeclareMathOperator{\Ext}{Ext}

\DeclareMathOperator{\ch}{ch}
\DeclareMathOperator{\Spec}{Spec}
\DeclareMathOperator{\sm}{sm}
\DeclareMathOperator{\td}{td}
\DeclareMathOperator{\Gr}{Gr}

\DeclareMathOperator{\SYZ}{SYZ}
\newcommand\srH{\mathscr{H}}
\newcommand*{\sheafhom}{\mathscr{H}\kern -.5pt om}
\newcommand\srA{\mathscr{A}}
\newcommand\srC{\mathscr{C}}
\newcommand\srX{\mathscr{X}}
\newcommand\srY{\mathscr{Y}}
\newcommand\srM{\mathscr{M}}
\newcommand\srN{\mathscr{N}}
\newcommand\srE{\mathscr{E}}
\newcommand\srF{\mathscr{F}}
\newcommand\srD{\mathscr{D}}

\subjclass[2020]{14C25, 14F08, 14J28, 14J42}

\begin{document}

\title{Filtrations on the derived category of twisted K3 surfaces}
\date{\today}

\author{Zaiyuan Chen}
\address{Zaiyuan Chen, Shanghai Center for Mathematical Sciences, Fudan University, Jiangwan Campus, Shanghai, 200438, China}
\email{zaiyuanchen16@fudan.edu.cn}

\author{Zhiyuan Li}
\address{Zhiyuan Li, Shanghai Center for Mathematical Sciences, Fudan University, Jiangwan Campus, Shanghai, 200438, China}
\email{zhiyuan\_li@fudan.edu.cn}

\author{Ruxuan Zhang}
\address{Ruxuan Zhang,  Beijing International Center for Mathematical Research, Peking University, No. 5 Yiheyuan Road Haidian District, Beijing, P.R.China 100871}
\email{rxzhang@pku.edu.cn}

\author{Xun Zhang}
\address{Xun Zhang, Morningside Center of Mathematics, Chinese Academy of Sciences, No. 55, Zhongguancun East Road, Haidian District, Beijing 100190}
\email{zhangxun2022@amss.ac.cn}

\maketitle

\begin{abstract}
We introduce and study the Shen-Yin-Zhao filtration on derived categories of twisted K3 surfaces. A main contribution is the construction of a twisted Beauville-Voisin class $\mathfrak{o}_{\mathscr{X}} \in \operatorname{CH}_0(X)$ that extends fundamental results of O'Grady and Shen-Yin-Zhao \cite{OG13, SYZ20} to twisted settings. This class enables:
\begin{itemize}
    \item A derived equivalence-invariant filtration $\mathbf{S}_\bullet(\mathrm{D}^{(1)}(\mathscr{X}))$ preserved under Fourier-Mukai transforms,
    \item A birational invariant filtration $\mathbf{S}^{\mathrm{SYZ}}_\bullet \operatorname{CH}_0$ on Bridgeland moduli spaces.
\end{itemize}
We prove $\mathbf{S}^{\mathrm{SYZ}}_\bullet \operatorname{CH}_0$ coincides with Voisin's filtration $\mathbf{S}^{\mathrm{BV}}_\bullet\operatorname{CH}_0$ (Theorem \ref{mainthm3}), providing a canonical candidate for the conjectural Beauville-Voisin filtration. Applications include Bloch's conjecture for (anti)-symplectic automorphisms  and  existence of algebraically coisotropic subvarieties.

\end{abstract}

\section{Introduction}

\subsection{Filtrations on K3 category}
Let $X$ be a K3 surface and let $\fro_X\in \CH_0(X)$ be the Beauville-Voisin class defined in \cite{BV04}. O'Grady \cite{OG13} introduced a filtration $\mathbf{S}_\bullet(X)$ on $\CH_0(X)$ given by 
\[
\mathbf{S}_i(X) = \bigcup_{\substack{\text{effective} \\ \deg([z])=i}} \left\{ [z] + \mathbb{Z} \cdot \fro_X \right\}
\]
for all effective zero-cycles $z$ of degree $i$. Denote by $\rD^b(X)$ the bounded derived category of coherent sheaves on $X$. Shen-Yin-Zhao \cite{SYZ20}  used O'Grady's filtration to construct a natural increasing filtration on $\rD^b(X)$:
\[
\mathbf{S}_0(\rD^b(X)) \subset \mathbf{S}_1(\rD^b(X)) \subset \cdots \subset \mathbf{S}_i(\rD^b(X)) \subset \cdots
\]
where $\mathbf{S}_i(\rD^b(X))$ consists of objects $\cE \in \rD^b(X)$ satisfying $c_2(\cE) \in \mathbf{S}_i(X)$. They established that for any $\cE \in \rD^b(X)$,
\[
c_2(\cE) \in \mathbf{S}_{d(\cE)}(X), \quad \text{with} \quad d(\cE) = \tfrac{1}{2} \dim \Ext^1(\cE,\cE).
\]
More generally, consider a twisted K3 surface $\srX \to X$ in the sense of \cite{Lie07}, and denote by $\rD^{(1)}(\srX)$ the bounded derived category of 1-fold twisted coherent sheaves (see \S \ref{ssec:twisted}). The existence of an analogous filtration on $\rD^{(1)}(\srX)$ is expected. The main result of this paper is:

\begin{theorem}\label{mainthm}
There exists a unique degree-1 cycle class $\fro_{\srX} \in \CH_0(X)$ such that for any $\cE \in \rD^{(1)}(\srX)$,
\[
c_2(\cE) - \rank(\cE) \fro_{\srX} \in \mathbf{S}_{d(\cE)}(X),
\]
where $d(\cE) = \tfrac{1}{2} \dim \Ext^1(\cE,\cE)$.
\end{theorem}
We call $\fro_{\srX}$ the $\srX$-twisted Beauville-Voisin class on $X$. When the twist $\srX \to X$ is essentially trivial, our construction yields $\fro_{\srX} = \fro_X$, recovering the main result of \cite{SYZ20}. We conjecture that $\fro_{\srX} = \fro_X$ holds for any twist $\srX \to X$. There are supporting evidence coming from the Bloch-Beilinson conjecture: when $X$ is defined over $\bar{\mathbb{Q}}$, Remark \ref{rmk:descent} shows that $\fro_{\srX}$ is also defined over $\bar{\mathbb{Q}}$, implying $\fro_{\srX} = \fro_X$ by the Bloch-Beilinson conjecture.

There are several useful consequences of Theorem \ref{mainthm}. First, we generalize \cite[Corollary 2.6]{Huy10}:
\begin{corollary}\label{cor:rigid}
For any two rigid objects $\cE, \cF \in \rD^{(1)}(\srX)$ with identical Mukai vectors,
\[
c_2(\cE) = c_2(\cF) \in \CH_0(X).
\]
\end{corollary}
Motivated by \cite{SYZ20}, we can define a derived filtration on twisted K3 categories:
\begin{equation}\label{eq:SYZ-filtration}
    \mathbf{S}_i(\rD^{(1)}(\srX)) = \left\{ \cE \in \rD^{(1)}(\srX) \mid c_2(\cE) - \rank(\cE)\fro_{\srX} \in \mathbf{S}_i(X) \right\}.
\end{equation}
This filtration exhibits natural invariance under derived equivalences:

\begin{theorem}\label{mainthm2}
Let $\Psi: \rD^{(1)}(\srX) \xrightarrow{\sim} \rD^{(1)}(\srY)$ be a derived equivalence between twisted K3 surfaces $\srX \to X$ and $\srY \to Y$. Then for any object $\cE \in \rD^{(1)}(\srX)$ and index $i$,
\[
\cE \in \mathbf{S}_i(\rD^{(1)}(\srX)) \iff \Psi(\cE) \in \mathbf{S}_i(\rD^{(1)}(\srY)).
\]
In particular, the filtration \eqref{eq:SYZ-filtration} is preserved under derived equivalences.
\end{theorem}
This extends \cite[Corollary 2]{SYZ20} to the twisted setting.

\subsection{BV  filtration on Bridgeland moduli spaces} 
The proof of Theorems \ref{mainthm} and \ref{mainthm2} involves a detailed study of the Beauville-Voisin filtration on the moduli space of stable objects in $\rD^{(1)}(\srX)$. Let $\mathbf{v} \in \widetilde{\rH}(\srX)$ be a primitive Mukai vector and $\sigma$ a $\mathbf{v}$-generic stability condition on $\rD^{(1)}(\srX)$. The coarse moduli space $M = M_\sigma(\srX,\mathbf{v})$ of $\sigma$-stable objects in $\rD^{(1)}(\srX)$ with Mukai vector $\mathbf{v}$ is a smooth projective hyper-K\"ahler variety. 

Voisin introduced a filtration $\mathbf{S}^{\mathrm{BV}}_{\bullet} \CH_0(M) \subseteq \CH_0(M)$ defined by
\[
\mathbf{S}_{i}^{\mathrm{BV}}\CH_0(M) := \left\langle x \in M \mid \dim \mathcal{O}_x \geq d - i \right\rangle,
\]
where $\mathcal{O}_x = \{ x' \in M \mid [x] = [x'] \in \CH_0(M) \}$. This filtration depends only on the isomorphism class of $M$.

On the other hand, there is a natural filtration 
\[
\mathbf{S}^{\mathrm{SYZ}}_{\bullet} \CH_0(M) \subseteq \CH_0(M)
\]
given by restricting the filtration $\mathbf{S}_{\bullet}(\rD^{(1)}(\srX))$ to $\CH_0(M)$, i.e.,
\begin{equation}\label{eq:SYZ}
    \mathbf{S}^{\mathrm{SYZ}}_{i}\CH_0(M) = \left\langle \cE \in M \mid \cE \in \mathbf{S}_i(\rD^{(1)}(\srX)) \right\rangle.
\end{equation}
A natural expectation, analogous to the untwisted case (\cite[Theorem 1.1]{LiZ22}), is that
\[
\mathbf{S}^{\mathrm{SYZ}}_{\bullet}\CH_0(M) = \mathbf{S}^{\mathrm{BV}}_{\bullet}\CH_0(M).
\]
Another goal of this paper is to verify this equality for arbitrary $M$. Our main technical result is:

\begin{theorem}\label{mainthm3}
Assume $\mathbf{v}^2 = 2n - 2 \geq 0$. Then for any object $\cE \in M$,
\[
\dim \mathcal{O}_{\cE} \geq n - i \quad \iff \quad \cE \in \mathbf{S}_i(\rD^{(1)}(\srX)).
\]
In particular, $\mathbf{S}_i^{\mathrm{BV}}\CH_0(M) = \mathbf{S}_i^{\mathrm{SYZ}}\CH_0(M)$.
\end{theorem}

This is proved by studying degeneracy loci of morphisms between twisted sheaves, generalizing the work in \cite{Voi15} and \cite{LiZ22}. Our method also applies to hyper-K\"ahler varieties of generalized Kummer type (see \S\ref{rmk:kummer}). A consequence is:

\begin{corollary}\label{cor：bir-invariant}
Let $Y$ be a hyper-K\"ahler variety of $K3^{[n]}$-type. If its algebraic Markman-Mukai lattice contains an isotropic vector, then $\bS_\bullet^{\mathrm{BV}}\CH_0(Y)$ is a birational invariant, i.e. for any birational map $f \colon Y \dashrightarrow Y'$ with $Y'$ hyper-K\"ahler,
\[
f_\ast (\bS_\bullet^{\mathrm{BV}} \CH_0(Y)) = \bS_\bullet^{\mathrm{BV}}\CH_0(Y').
\]
In particular, if $\operatorname{rank}(\NS(Y)) \notin \{2, 3\}$, then $\bS_\bullet^{\mathrm{BV}}\CH_0(Y)$ is a birational invariant. 
\end{corollary}

This filtration plays an important role in studying Bloch's conjecture for zero-cycles on hyper-K\"ahler varieties. An advantage of the Shen-Yin-Zhao filtration $\bS_\bullet^{\mathrm{SYZ}}\CH_0 = \bS_\bullet^{\mathrm{BV}}\CH_0$ is that its graded pieces are controlled by the first one. This yields

\begin{theorem}\label{thm:bloch}
Let $Y$ be a smooth projective hyper-K\"ahler variety of $K3^{[n]}$-type with a birational Lagrangian fibration. Then $\CH_0(Y)$ can be equipped with a natural filtration in \eqref{eq:SYZ} and  Bloch's conjecture holds  for  any (anti-)symplectic birational automorphism  preserving the birational Lagrangian fibration, i.e.
\begin{equation}
    \phi^\ast|_{\rH^{2,0}(Y)}=\pm \id \Leftrightarrow  \phi_\ast|_{\mathrm{Gr}_s\CH_0(Y)}=(\pm 1)^s\id,~\forall ~1\leq s\leq n.
\end{equation}
where $\mathrm{Gr}_s\CH_0(Y)$ is the $s$-th graded piece of the the filtration $\bS_\bullet^{\rm BV}\CH_0(Y)$.
\end{theorem}

Another application concerns algebraically coisotropic subvarieties. Adapting  the same argument in \cite[Theorem 0.5]{SYZ20} gives:
\begin{corollary}\label{cor:coiso}
Let $M$ be a Bridgeland moduli space as above. For $0 \leq i \leq \frac{\mathbf{v}^2 + 2}{2}$, there exist algebraically coisotropic subvarieties $\iota \colon Z_i \dashrightarrow B_i$ with $Z_i \subseteq M$ of codimension $i$, whose general fibers are constant cycle subvarieties of dimension $i$.
\end{corollary}

\subsection{Organization of the paper}
In Section \ref{sec:preliminary}, we establish foundational concepts for twisted K3 surfaces $\mathscr{X} \to X$, including $\mu_m$-gerbes, twisted sheaves of weight $d$, and derived categories $\mathrm{D}^{(d)}(\mathscr{X})$. We develop the twisted Chern character $\operatorname{ch}_{\mathscr{X}} : K_0(\mathscr{X}) \to \mathrm{CH}^*(X)$ and study the  properties in Lemma \ref{c_2(i_*L) and i_*(c_1(L))}. The twisted Mukai lattice $\widetilde{\mathrm{H}}(\mathscr{X}) = \exp(\mathbf{B}) \cdot \widetilde{\mathrm{H}}(X)$ and Bridgeland stability conditions are systematically reviewed.

Section \ref{sec:torsion} analyzes cycles on moduli spaces $M_H(\mathscr{X},\mathbf{v})$ of Gieseker stable $\mathscr{X}$-twisted torsion sheaves with primitive Mukai vector $\mathbf{v} = (0,L,s)$. We prove the Lagrangian fibration $M_H(\mathscr{X},\mathbf{v}) \to |L|$ admits dense constant cycle Lagrangian subvarieties (Proposition \ref{prop:dense-lag}), and establish the correspondence $c_2(\mathcal{E}) \in \mathbf{S}_{\mathrm{d}(\mathcal{E})}(X)$ for $\mathcal{E} \in M_H(\mathscr{X},\mathbf{v})$ (Proposition \ref{prop:tor-fil}). Key to this is Theorem \ref{thm:torsion}, which shows:  
$$c_2(\mathcal{E}) \in \mathbf{S}_i(X) \iff \dim(\mathcal{O}_{\mathcal{E}} \cap M^{\mathrm{sm}}_{\mathscr{X}}) \geq \mathrm{d}(\mathcal{E}) - i$$
via finite incidence correspondences between smooth loci of moduli spaces for different twists.

Section \ref{sec:twisted BV} presents our core construction: the twisted Beauville-Voisin class $\mathfrak{o}_{\mathscr{X}} \in \mathrm{CH}_0(X)$ defined via Chern classes of stable twisted locally free sheaves on constant cycle Lagrangian subvarieties (Definition \ref{twisted:o}). Using degeneracy loci of morphisms between polystable sheaves (Definition \ref{def:deg}), we prove $\mathfrak{o}_{\mathscr{X}}$ is well-defined and independent of choices (Proposition \ref{prop:main}). This culminates in Theorem \ref{mainthm}:  
$$c_2(\mathcal{E}) - \operatorname{rank}(\mathcal{E})\mathfrak{o}_{\mathscr{X}} \in \mathbf{S}_{\mathrm{d}(\mathcal{E})}(X)$$
for all $\mathcal{E} \in \mathrm{D}^{(1)}(\mathscr{X})$, generalizing [23] to twisted surfaces.

Sections \ref{sec:proof of mainthm2} and \ref{sec:bloch} complete the main results:  
\begin{enumerate}
    \item [-] Theorem \ref{mainthm2} (derived invariance of $\mathbf{S}_i(\mathrm{D}^{(1)}(\mathscr{X}))$ via Beauville-Voisin groups (Proposition \ref{prop:BV})  
    \item [-] Theorem \ref{mainthm3} ($\dim \mathcal{O}_{\mathcal{E}} \geq n-i \iff \mathcal{E} \in \mathbf{S}_i(\mathrm{D}^{(1)}(\mathscr{X}))$ unifying SYZ and BV filtrations)  
    \item [-] Theorem \ref{thm:bloch} (Bloch's conjecture for (anti)-symplectic birational maps), proved using birational motives (Lemma \ref{lem:generator}) and the identity $\mathfrak{h}^\circ(Y) \cong \mathfrak{h}^\circ(X^{[n]})$ (Proposition \ref{prop:birationalmotive}).
\end{enumerate}

In the appendix, we extend Bayer-Macri's result  in \cite{BM} to twisted K3 surfaces, proving that Bridgeland moduli spaces $M_\sigma(\mathscr{X},\mathbf{v})$ are isomorphic to moduli of Gieseker stable locally free twisted sheaves (Theorem \ref{thm:locally-free}).

\subsection*{Acknowledgement}
We  want  to  thank Ziyu Zhang and Junliang Shen for helpful discussions. Z. Chen and R. Zhang are supported by the NKRD Program of China (No. 2020YFA0713200) and NSFC grant (No. 12121001). Z. Li is supported by NSFC grant (No. 12171090 and No. 12425105) and Shanghai Pilot Program for Basic Research (No. 21TQ00). Z. Li is also a member of LMNS. 

\subsection*{Convention} In this paper, we use $\CH^\bullet(X)$ or $\CH_\bullet(X)$ to denote the Chow groups of $X$ with rational coefficients.

\section{Preliminary}\label{sec:preliminary}
Throughout this section, we work over  $\CC$. 

\subsection{Twisted K3 surface and twisted sheaves}
\label{ssec:twisted}Let $p:\mathscr{X}\rightarrow X$ be a $\mu_m$-gerbe over a smooth projective K3 surface $X$. This corresponds to a pair $(X,\alpha)$ for some  $\alpha\in \rH^2_{\fl}(X,\mu_m)$, where the cohomology group is with respect to the flat topology. Recall that there is a Kummer exact sequence 
\begin{equation}\label{eq:Kummer}
 1\rightarrow \mu_m\rightarrow {\cO_X^\times} \xrightarrow{x\mapsto x^m} {\cO_X^\times}\rightarrow 1
\end{equation}
in flat topology and it induces a surjective map 
\begin{equation}\label{braumap}
\rH^2_{\rm fl}(X,\mu_m)\rightarrow \mathrm{Br}(X)[m]. 
\end{equation}
We denote by $[\srX]$ the image of  $\alpha$ in $\Br(X)[m]$.  We say  $\srX\to X$  is {\it essentially trivial} if the associated Brauer class $[\srX]$ is zero in $\Br(X)$. 

%\textcolor{red}{
%A $d$-fold $\srX$-{\it twisted sheaf}  on $p:\srX\to X$ is an $\cO_{\srX}$-module $\cF$ compatible with the $\mu_m$-gerbe structure } ({\it cf}.~\cite[Def 2.1.2.2]{Lie07}). The $\mu_m$-gerbe structure gives a natural $\mu_m$-action on a $\cO_{\srX}$-module on $\srX$. 
A $d$-fold $\srX$-{\it twisted sheaf}  on $p:\srX\to X$ is an $\cO_{\srX}$-module of weight $d$ with respect to the $\mu_m$-action ({\it cf}.~\cite[Def 2.1.2.2]{Lie07}). 
Let $\rD^{(d)}(\srX)$ be the bounded derived category of coherent $d$-fold $\srX$-twisted sheaves.  For a $\srX$-twisted locally free sheaf $\cE\in \rD^{(1)}(\srX)$,  we have a diagram
\begin{equation*}
    \begin{tikzcd}
\PP(\cE) \arrow[r, "q"] \arrow[d, "\varpi"] & Y \arrow[d, "\pi"] \\
\srX \arrow[r, "p"]                         & X          ,
\end{tikzcd}
\end{equation*}
where $\PP(\cE)\to Y$ is essentially trivial and $Y$ is served as a Severi-Brauer variety of $(\srX)$.

Moreover, there is a  derived tensor product  $$-\otimes^{\LL}-: \rD^{(d)}(\srX) \times \rD^{(d')}(\srX)\to \rD^{(d+d')}(\srX)$$
and the derived dual functor
$$(-)^\vee:\rD^{(d)}(\srX)\to \rD^{(-d)}(\srX)^{\rm opp}.$$
When $d=0$, one may identify  $\rD^{(0)}(\srX)=\rD^{b}(X)$ via the pullback functor.

\subsection{Chern character of twisted sheaves}
Let $K_0(\srX)$ be the Grothendieck group of $1$-fold $\srX$-twisted coherent sheaves. In \cite[\S 3.3.4]{LMS14},  Lieblich-Maulik-Snowden have introduced a twisted Chern character 
\begin{equation}
\ch_\srX:    K_0(\srX)\to \CH^\ast(X),
\end{equation}
which takes values in the Chow ring of $X$. For a locally free sheaf $\cF$, it is defined by 
\begin{equation}\label{eq:twist-chern}
   \ch_\srX(\cF)=\sqrt[m]{\ch(p_\ast \cF^{\otimes m})} 
\end{equation}
where the $m$-th root is uniquely determined by requiring $r(\ch_\srX(\cF))=\rank(\cF)$. One can define the twisted Chern class $$c_i:K_0(\srX)\to \CH^i(X)$$ accordingly. The definition above avoids the use of the Chow groups of stacks. However, one can compare it with the Chern character $\ch:K_0(\srX)\to \CH^\ast(\srX)$ for stacks. An easy fact is 

\begin{lemma}\label{c_2(i_*L) and i_*(c_1(L))}
For any $x\in K_0(\srX)$, we have $p_\ast\ch(x)=\frac{1}{m}\ch_\srX(x)$.
\end{lemma}
\begin{proof}
   It suffices to check the case when $x$ is the class of a locally free sheaf $\cF$. 
    
    According to \cite{Vistoli89}, the pullback map 
    $$p^\ast:\CH^\ast(X)\to \CH^\ast(\srX)$$
    is a ring isomorphism. Hence there is a unique $z \in \CH^\ast(X)$ such that $\ch(\cF)=p^*z$. Then we have
    $$p_*(\ch(\cF)^m)=p_*p^*z^m=\dfrac{1}{m}z^m.$$
    On the other hand, from the definition \eqref{eq:twist-chern},  we have 
    $$p^*(\ch_\srX(\cF)^m)=p^*\ch(p_\ast \cF^{\otimes m})=\ch(p^*p_\ast \cF^{\otimes m})=\ch(\cF^{\otimes m})=\ch(\cF)^m.$$
Here we use the fact that $p^\ast \circ p_\ast$ is identity for $0$-twisted sheaves.  The assertion then follows from the projection formula on Chow groups.  
 \end{proof}
 
 \begin{example}\label{ex:torsion}(Chern class of torsion sheaves)
   Let us consider the Chern class of a twisted torsion sheaf whose schematic support is an irreducible curve $C$ on $X$. There is a Cartesian diagram 
\begin{equation*}
\begin{tikzcd}
\srC \ar[r," \Phi"] \ar[d,"q"]& \srX \ar[d,"p"] \\
C\ar[r,"\phi"] &  X.
\end{tikzcd} 
\end{equation*}
Let $\cL$ be a 1-fold $\srC$-twisted coherent sheaf. For the $\srX$-twisted torsion sheaf $\Phi_\ast \cL$, we have
\begin{equation}
\begin{aligned}
  \frac{1}{m}  \ch_\srX(\Phi_\ast\cL)&=p_\ast \ch(\Phi_\ast\cL)\\ &=p_\ast(\Phi_\ast(\ch(\cL)\cdot \mathrm{td}_{\srC/\srX}))\\
    &=\phi_\ast q_\ast(\ch(\cL)\cdot \mathrm{td}_{\srC/\srX})\\
    &=\phi_\ast (q_\ast(\ch(\cL))\cdot \mathrm{td}_{C/X})=\frac{1}{m}\phi_\ast \ch_{\srC}(\cL)\cdot (1-\frac{1}{2}[C]).
\end{aligned}
\end{equation} 
In particular, we have $c_2(\Phi_\ast\cL)=[C]^2-\phi_\ast (c_1(\cL))\in \CH_0(X)$. 
It has the same form as in the untwisted case. The foregoing argument uses the  Grothendieck-Riemann-Roch theorem on $\mu_m$-gerbes (\cf.~\cite[Proposition 2.2.7.8]{Lie07}). Alternatively one can obtain the above formula using the same method as  \cite[Lemma 4.1.4]{Bragg18}.
\end{example}
 
 Similarly, we can define the topological Chern characters $\ch_\srX^{\rm top}$ of elements in $K_0(\srX)$.  Moreover, let us briefly review the  construction of twisted Mukai lattice of $\srX\to X$.  Recall that there is a {\it B-field} $$\bB\in \rH^2_{\rm tr}(X,\QQ):=\NS(X)^\perp$$ such that $[\srX]=\delta(\bB)$, where  $\delta :\rH^2(X,\QQ)\to \rH^2(X, \cO_X^\times)$ is induced from the exponential map.  
 We define $$\widetilde{\rH}(X)=\ZZ\oplus \rH^2(X,\ZZ) \oplus \ZZ$$ to be the (untwisted) Mukai lattice under the Mukai pairing $\langle,\rangle$. We use $\widetilde{\rH}_{\rm alg}(X)=\ZZ\oplus \NS(X)\oplus \ZZ$ to denote the algebraic part of $\widetilde{\rH}(X)$. There is a rational isometry given by 
$$\begin{aligned}
    \exp(\bB):\widetilde{\rH}(X)\otimes\QQ&\longrightarrow \widetilde{\rH}(X)\otimes \QQ \\
    (r,L,s) &\mapsto (r,L+rB,s+\frac{rB^2}{2}).
\end{aligned}
$$
\begin{definition}
   We define the {\it twisted Mukai lattice of $\srX\to X$} as
\[
\widetilde{\rH}(\srX) \coloneqq \exp(\bB)\cdot \widetilde{\rH}(X) \subset \widetilde{\rH}(X) \otimes_{\mathbb{Z}} \mathbb{Q},
\]
Up to a lattice isometry,  $\widetilde{\rH}(\srX)$ is independent of the choice of a $B$-field lift. One can similarly define the twisted algebraic Mukai lattice $\widetilde{\rH}_{\rm alg}(\srX)$ by replacing $\widetilde{\rH}(X)$ with $\widetilde{\rH}_{\rm alg}(X)$.  For $\cF\in \rD^{(1)}(\srX)$, its  Mukai vector is defined by
$$\bfv^B(\cF)=\exp(\bB)\ch^{\rm top}_{\srX}(\cF) \sqrt{\td_X^{\rm top}} \in \widetilde{\rH}(\srX).$$
\end{definition}
In this paper, we will  fix a $B$-field $\bB$ and simply write it as $\bfv(\cF)$.

\subsection{Bridgeland moduli space}
Let us fix some notations. Recall that a Bridgeland stability condition $\sigma$ on $\rD^{(1)}(\srX)$ consists of a pair $(Z, A)$, where $Z \colon K_{\mathrm{num}}(\srX) \to \mathbb{C}$ is a central charge and $A \subset \rD^{(1)}(\srX)$ is the heart of a bounded $t$-structure, satisfying certain properties (see \cite[Section 2]{BM}). 

For $\phi \in \mathbb{R}$, denote by $P(\phi)$ the full subcategory of objects $\cE$ with $Z(\cE) \in \mathbb{R}_{>0} \cdot e^{\pi i\phi}$. Then $P$ is called a slicing of $\rD^{(1)}(\srX)$, and $\phi$ is called the phase of $\cE$. The stability condition $\sigma$ is called geometric if there exist $\omega, \beta \in \mathrm{NS}(\srX) \otimes \mathbb{Q}$ with $\omega$ ample, such that
\begin{align*}
    Z = Z_{\omega, \beta} \colon K_{\mathrm{num}}(\srX) &\longrightarrow \mathbb{C}, \\
    \cF &\mapsto \langle \exp(\beta + i\omega), \mathbf{v}(\cF) \rangle,
\end{align*}
and the heart $A_{\omega, \beta} = \langle F_{\omega, \beta}, T_{\omega, \beta}[1] \rangle$ is obtained by tilting with respect to the torsion pair $(T_{\omega, \beta}, F_{\omega, \beta})$, where
\begin{align*}
    T_{\omega, \beta} &= \left\{ \cE \in \rD^{(1)}(\srX) : \parbox{8.5cm}{\raggedright any semistable factor $\cF$ of $\cE$ satisfies $\mu_{\omega, \beta}(\cF) > 0$} \right\}, \\
    F_{\omega, \beta} &= \left\{ \cE \in \rD^{(1)}(\srX) : \parbox{8.5cm}{\raggedright any semistable factor $\cF$ of $\cE$ satisfies $\mu_{\omega, \beta}(\cF) \leq 0$} \right\}.
\end{align*}
Here, $\mu_{\omega, \beta}(\cF) = \frac{c_1(\cF) \cdot \omega}{\rank(\cF)}$ denotes the twisted slope. The set of Bridgeland stability conditions is denoted by $\mathrm{Stab}(\srX)$, equipped with a natural metric topology. There is a distinguished connected component containing all geometric stability conditions (\cite[Theorem 2.10]{BM14}), denoted by $\mathrm{Stab}^{\dagger}(\srX)$.

Let $\bfv=(r,L,s)\in\widetilde{\rH}(\srX)$ be a primitive Mukai vector with $\bfv^2\geq 0$ and let $\sigma$ be a $\bfv$-generic Bridgeland stability condition. Let $\srM_\sigma(\srX, \bfv)$ be the moduli stack  of $\sigma$-stable objects in $\rD^{(1)}(\srX)$ with Mukai vector $\bfv$.  Let $\srM_\sigma(\srX,\bfv)\to M_\sigma(\srX,\bfv)$ be the coarse moduli map, which can be also viewed as a gerbe on $M_\sigma(\srX,\bfv)$.  In this way, we have  
 a universal object $$\srE\in \rD^{(-1,1)}(\srX\times \srM_\sigma(\srX, \bfv)),$$
on $\srX\times \srM_\sigma(\srX,\bfv)$.

According  to  \cite[Sections 7 and 8]{Yos01}, \cite[Proposition 14.2]{Bridgeland08} and \cite[Theorem 6.10, Section 7]{BM14}, when $\bfv^2\geq -2$, $M_\sigma(\srX, \bfv)$  is a non-empty smooth hyper-K\"ahler variety of K3$^{[n]}$-type of  dimension $2+v^2$. If $\sigma$ is a Gieseker stability condition with respect to an ample line bundle $H\in \Pic(X)$, we may write it as  $\srM_H(\srX,\bfv)$. When $r\geq 0$ and $\bfv^2\geq -2$, the moduli space $\srM_H(\srX,\bfv)$ is not empty for a $\bfv$-generic $H$. There is a natural isometry 
  \begin{equation}\label{eq:FM}
  \begin{aligned}
 \bfv^\perp \xrightarrow{\cong}  \NS(M_\sigma(\srX,\bfv))
  \end{aligned}
  \end{equation}
 induced by the (quasi-)universal object,  where $\bfv^\perp$ is the orthogonal complement of $\bfv$ in $\widetilde{\rH}_{\rm alg}(\srX)$ (\cf.~\cite[Sections 7 and 8]{Yos01}).

In the study of zero cycles on $M_\sigma(\srX,\bfv)$, a very important result proved by Marian-Zhao is 
\begin{theorem}[\cite{MZ20}]\label{thm:MZ}
    Let $\cE,\cF\in M_\sigma(\srX,\bfv)$. Then   $c_2(\cE)=c_2(\cF)\in \CH_0(X)$ if and only if $[\cE]=[\cF]\in \CH_0(M_{\sigma}(\srX,\bfv))$. 
\end{theorem}

We also need the square zero conjecture for hyper-K\"ahler varieties of $K3^{[n]}$-type. 
\begin{theorem}[{\cite[Theorem 1.1, Theorem 1.5]{BM14}}]\label{thm:BM}
    Notations as above. Assume that there exists an isotropic class $\bfw\in\bfv^\perp$, then $M_\sigma(\srX,\bfv)$ admits a birational Lagrangian fibration.
\end{theorem}

\section{Cycles on HK with birational Lagrangian fibrations}\label{sec:torsion}

 In this section, we investigate the zero cycles on  moduli space of twisted torsion sheaves. 
\subsection{Density of constant cycle Lagrangian subvarieties}
Let us recall some general results for cycles on hyper-K\"ahler varieties with a Lagrangian fibration.  The first result is a strengthening of \cite[Theorem 1.3]{Lin18}
. \begin{proposition}\label{prop:dense-lag}
    If a hyper-K\"ahler variety $M$ admits a birational Lagrangian fibration, then there is a dense collection of constant cycle Lagrangian subvarieties. Moreover, points on those constant cycle Lagrangian subvarieties represent the same class.
 \end{proposition}
 
 \begin{proof}

As the  property is preserved under birational transformations, 
 we are reduced to considering the case that $M$ admits a Lagrangian fibration $M\to \PP^n$. If $T\to \PP^n$ is a morphism, we set $$M_T:=M\times_{\PP^n} T.$$  Let $U\subseteq \PP^n$ be the open locus parametrizing all the smooth fibers of $M\to \PP^n$. According to  \cite[Theorem 1.3]{Lin18}, there is a  constant cycle Lagrangian subvariety $\Sigma\subseteq M$ such that the restriction $\Sigma_U=\Sigma\cap M_U\to U$ is finite. This gives a map 
 \begin{equation}\label{eq:mult-pt}
     \Spec(\CC(\Sigma))\to\Spec(\CC(\PP^n))\to \PP^n. 
 \end{equation}  

 Let $\eta$ be the generic point of $\PP^n$ and let $M_\eta$ be the generic fiber of $M\to \PP^n$.  The multisection $\Sigma_U$ defines a closed point $0_\Sigma\in M_{\eta}(\CC(\Sigma))$ via \eqref{eq:mult-pt}.  Now we may consider the torsion points on the abelian variety  $(M_{\bar{\eta}}, 0_\Sigma)$.  For each geometrically closed torsion point $p\in M_{\bar{\eta}}$, it can be spread out to be a multisection on the abelian fibration $M_U\to U$. Its Zariski closure $\cZ_p\subseteq M$ is  a constant cycle Lagrangian subvariety on $M$. As the geometrically closed torsion points of $M_{\bar{\eta}}$  are dense and $\CH_0(M)$ is torsion-free, 
the assertion follows. 
 \end{proof}

 \begin{remark}
 From the construction, one can see that all these constant cycle Lagrangian subvarieties are contained in a single  orbit $O_x$ with $x\in \Sigma$.     
 \end{remark}
\subsection{Twisted torsion sheaves with maximal dimensional rational equivalent orbit} 

 Let $\bfv=(0,L,s) \in \widetilde{\rH}(X)$ be a primitive vector with $L\in \Pic(X)$ effective. Let $H$ be $v$-generic. The moduli space $M_H(\srX,\bfv)$ parameterizes torsion sheaves of pure dimension $1$.
 Their schematic support are  curves in $|L|$. Moreover, if the support curve $\srC\to C$ of $\cE\in M_H(\srX,\bfv)$ is integral, then $\rank(\cE|_\srC)=1$ since $c_1(\cE)=[C]$.    
 
  There is a Lagrangian fibration 
  \begin{equation}
    \pi_\srX:  M_H(\srX,\bfv)\to |L|.
  \end{equation}
by sending $\cE$ to its  schematic support. 

\begin{proposition}\label{prop:tor-fil}
    For $\cE\in M_H(\srX,\bfv)$, we have $c_2(\cE)\in \bS_{\rd(\cE)}(X)$. 
\end{proposition}
\begin{proof}
   By Example \ref{ex:torsion},  $c_2(\cE)$ is a zero cycle supported on a curve $C\in |L|$. Since the genus of $C$ is less than or equal to $\frac{1}{2}(L^2+2)=\rd(\cE)$, the assertion follows from \cite[Claim 0.2]{OG13} (see also \cite[Proposition 15]{Voi15}).
\end{proof}

As in \cite{Voi15} and \cite{SYZ20},  set $\dim  M_\sigma(\srX,\bfv)=2n$ and we can define the incidence 
\begin{equation}
    \Gamma':=\Big\{(\xi,\cE)~|~c_2(\cE)=[\xi]+k\fro_X\in \CH_0(X)\Big\}\subset X^{[n]}\times M_\sigma(\srX,\bfv),
\end{equation}
 where $k\in \QQ$ is a constant determined by $\bfv$. It is known that $\Gamma'$ is a countable union of closed subsets of $X^{[n]}\times M_\sigma(\srX,\bfv)$ and   there are two projections
\begin{equation*}
    \pi_1:\Gamma'\to M_\sigma(\srX, \bfv),\quad  \pi_2:\Gamma'\to X^{[n]}.
\end{equation*}
 Due to Proposition \ref{prop:tor-fil}, $\pi_1$ is dominant. By the same argument in \cite[Proposition 1.3]{OG13},  we know that $\pi_2$ is also dominant.
 
 By cutting $\Gamma'$ via  hyperplane sections, we can obtain a  subvariety $\Gamma \subseteq \Gamma'$ such that  $\Gamma$ is generically finite over $M_\sigma(\srX,\bfv)$ and $X^{[n]}$. Then we have 
\begin{equation}
    c_2(\cE)\in\bS_i(X)\Rightarrow \dim \cO_\cE\geq \rd(\cE)-i.
\end{equation} 
 for any $\cE\in M_\sigma(\srX,\bfv)$. This yields

\begin{corollary}\label{cor:torsion-BV}
 Let $\Sigma \subseteq M_\sigma(\srX,\bfv)$ be a  constant cycle  Lagrangian subvariety in Proposition \ref{prop:dense-lag}. Then $c_2(\cE)\in \bS_0(X)$ for any $\cE\in \Sigma$.
\end{corollary}
\begin{proof}
  
  According to the construction in the proof of Proposition \ref{prop:dense-lag},  we get a Lagrangian constant cycle subvariety $\Sigma_\srX'$ such that for any $\cE'\in \Sigma_\srX'$, we have $c_2(\cE')=c_2(\cE)$. 

  Moreover, we have the density of such Lagrangian constant cycle subvarieties. Hence we may assume that $\dim \pi_2(\pi_1^{-1}(\Sigma_\srX'))=d(\cE)$. According to Theorem \ref{thm:MZ}, $\pi_2(\pi_1^{-1}(\Sigma_\srX'))$ is a Lagrangian constant cycle subvariety in $X^{[n]}$.

Then for any $\cF\in \pi_2(\pi_1^{-1}(\Sigma_\srX'))$, we have $c_2(\cF)\in \bS_0(X)$ by \cite[Theorem 9]{Voi15}. Therefore $c_2(\cE')\in \bS_0(X)$ by the definition of $\Gamma$ for any $\cE'\in \Sigma_\srX'$.
\end{proof}

\subsection{Correspondence between torsion sheaves}

With notations as above, let us write 
$$M_{\srX}:=M_H(\srX, \bfv) $$
for short.
Assume  $L\in \Pic(X)$  is effective. We denote by $\cC\to |L|$  the complete family of curves in $|L|$  and  $\cC^{\sm} \to |L|^{\sm}$ the family of smooth curves in $|L|$. Let $M^{\sm}_{\srX}=M_\srX\times_{|L|} |L|^{\sm}$. Then $M^{\sm}_{\srX} \to |L|^{\sm}$ is a torsor of the relative Jacobian $\Pic^{0}(\cC^{\sm}/ |L|^{\sm})$ with the action  $M^{\sm}_{\srX}\times_{|L|^{\sm}} \Pic^0(\cC^{\sm}/|L|^{\sm}) \to M^{\sm}_{\srX} $ given by 
\begin{equation*}
 (\cE, L)\mapsto \cE\otimes (i_C)_\ast L
\end{equation*}
where  $\cE\in M_\srX^{\sm}$ with $\mathrm{supp}(\cE)=C$ and $L\in \Pic^0(C)$
(\cf.~\cite[Subsection 3.2]{Huy}).

\begin{theorem}\label{thm:torsion}
   For  $\cE\in M^{\sm}_{\srX}$, we have 
    \begin{equation}c_2(\cE)\in \bS_{i}(X)\Leftrightarrow \dim (O_{\cE}\cap M^{\sm}_{\srX}(\bfv))\geq  \rd(\cE)-i.
    \end{equation}
\end{theorem}
\begin{proof}
If $\srX\to X$ is essentially trivial, this is established in \cite[Theorem 3.1 \& Theorem 1.1]{LiZ22}. If $\srX\to X$ is not essentially trivial,  we can establish a  finite correspondence between $M_X^{\sm}$ and $M_\srX^{\sm}$ which preserves the two  filtrations.
More precisely, consider the incidence variety
\begin{equation}
    R=\Big\{ (\cE,E)\in M^{\sm}_{\srX}\times M_X^{\sm}| ~c_2(\cE)-c_2(E)\in \bS_0(X)~\Big\},
\end{equation}
which is a countable union of closed subvarieties of $M^{\sm}_{\srX}\times M_X^{\sm}$. Then
\begin{claim}\label{claim}
  There exists a subvariety $W\subseteq R$ such that the projections $\pi_1:W \to M_\srX^{\sm}$ and $\pi_2:W \to M_X^{\sm}$ are finite and surjective.  
\end{claim}

If such $W$ exists,  when $\dim (O_{\cE}\cap M^{\sm}_{\srX})\geq  \rd(\cE)-i=\frac{\bfv^2+2}{2}-i$, we can find a subvariety $$Z\subset O_{\cE}\cap M^{\sm}_{\srX}$$ with  $\dim Z\geq  \rd(\cE)-i$. Then $Z'=\pi_2(\pi_1^{-1}(Z))$ is a constant cycle subvariety by Theorem \ref{thm:MZ}  and we have $\dim Z'=\dim Z \geq \frac{\bfv^2+2}{2}-i$.  This means $\dim O_{E}\geq \frac{\bfv^2+2}{2}-i$ for any $E\in  Z'$. According to the proof of \cite[Theorem 1.1]{LiZ22}, we have  $$c_2(E)\in \bS_i(X) $$ where $E\in Z'$.   From the construction of $R$, we have $c_2(\cE)-c_2(E)\in\bS_0(X)$ which implies $c_2(\cE)\in \bS_i(X)$.

\subsubsection*{Proof of Claim \ref{claim}} Let us explain the construction of $W$.
 Let $\Sigma_{\srX}\subseteq M_{\srX}$ and  $\Sigma_X\subseteq M_X$ be constant cycle Lagrangian subvarieties defined in the proof of Proposition \ref{prop:dense-lag}. Their restrictions to  the smooth locus $M_\srX^{\sm}$ (resp.~$M_X^{\sm}$) are multisections of the abelian fibration of degree $d_1$ (resp. $d_2$). 
 
 Let $\Sigma_{\srX,C}$ and $\Sigma_{X,C}$ be their restriction to the fiber over $C\in |L|^{\sm} $.  Denote by $i_C:C\hookrightarrow X$ the closed embedding. We can define 
 \begin{equation}
    W=\Big\{ (\cE, E)\in M_{\srX}^{\sm}\times M_X^{\sm}|~ (i_C)^\ast(\cE-\frac{1}{d_1}\Sigma_{\srX,C})=(i_C)^\ast( E-\frac{1}{d_2}\Sigma_{X,C}) \in \Pic^0_\QQ(C)\Big\}. 
 \end{equation}
By Example \ref{ex:torsion}, for $(\cE,E)\in W$, we have 
\begin{equation}
    \begin{aligned}
        c_2(\cE)-\frac{1}{d_1}c_2(\Sigma_{\srX,C})=c_2(E)-\frac{1}{d_2}c_2(\Sigma_{X,C}) \in \CH_0(X). 
    \end{aligned}
\end{equation} 
Here, $c_2(\Sigma_{\srX,C})=\sum\limits_{\cF\in \Sigma_{\srX,C}} c_2(\cF)$ and $c_2(\Sigma_{X,C})=\sum\limits_{F\in \Sigma_{X,C}} c_2(F)$. 

By Proposition \ref{cor:torsion-BV}, for every element  $\cF\in \Sigma_{\srX,C}$, we have $c_2(\cF)\in \bS_0(X)$. Thus $$\frac{1}{d_1} c_2(\Sigma_{\srX,C})=\frac{1}{d_2}c_2(\Sigma_{\srX,C})\in \bS_0(X).$$
Here, the equality holds because they have the same degree.   This forces  $c_2(\cE)=c_2(E)$.  Moreover, one can easily check that the projections are finite and surjective. Indeed, $W$ is obtained by taking the graph of the natural isomorphism between the two torsors $M_{\srX}^{\sm} \to |L|^{\sm}$ and $M_{X}^{\sm} \to |L|^{\sm}$ after the finite base change $$\Sigma_{\srX}^{{\sm}}\times_{|L|^{\sm}} \Sigma_X^{{\sm}} \to |L|^{{\sm}}.$$  
This completes the proof.

\end{proof}

\section{Twisted Beauville-Voisin class}\label{sec:twisted BV}

In this section, we construct the twisted Beauville-Voisin class on $\srX$. 
\subsection{Construction of $\fro_{\srX}$}Throughout this section, all the Mukai vectors are assumed to be primitive and stability conditions are generic. 
\begin{definition}\label{twisted:o}
Let $N=M_\sigma(\srX,\bfw)$ be a  Bridgeland moduli space and let $N^{\rm l.f.}\subseteq N$ be the open subset which consists of $\srX$-twisted locally free sheaves. Let $\cF\in N^{\rm l.f.}$ be a point lying on a constant cycle Lagrangian  subvariety $Z$. We define 
\begin{equation}
     \fro_{\srX}=\frac{c_2(\cF)}{\rank \cF}+(1-\frac{\deg(c_2(\cF))}{\rank \cF})\fro_X,
\end{equation}
which is a zero cycle of degree $1$. 
\end{definition}

\begin{remark}\label{rmk:tensor}
    The construction of $\fro_\srX$ is clearly invariant under the action of tensoring line bundles. This is because $c_2(\cF)-c_2(\cF\otimes L)\in \bS_0(X)$. 
\end{remark}

The existence of such $N$ and $\cF$ in ensured by the following result. 

\begin{proposition}\label{prop:existence}
 Let $\bfv\in\widetilde{\rH}(\srX)$ and  assume $\bfv^\perp\subseteq \widetilde{\rH}_{\rm alg}(\srX)$ contains an isotropic vector. Assume that $\rank(\bfv)=r>\frac{\bfv^2+2}{2}$ and there exists a slope stable element in $M_H(\srX,\bfv)$.  Then  $M_H(\srX,\bfv)$ contains a constant cycle Lagrangian subvariety $Z$ such that $Z\cap M_H^{\rm l.f.}(\srX,\bfv)\neq \emptyset$. 
\end{proposition}
\begin{proof}
By \eqref{eq:FM}, we know that $\NS(M_H(\srX,\bfv))\cong \bfv^\perp$ contains an isotropic element.   Then $M_H(\srX,\bfv)$ admits a birational Lagrangian fibration by Theorem \ref{thm:BM}.  By Proposition \ref{prop:dense-lag}, there is a dense collection of constant cycle Lagrangian subvarieties on $M_H(\srX,\bfv)$. So it suffices to show  $M_H^{\rm l.f.}(\srX,\bfv)\neq \emptyset$.

Note that under the condition $r>\frac{\bfv^2+2}{2}$, all $\mu_H$-stable $\srX$-twisted sheaves are locally free (\cf.~\cite[Remark 3.2]{Yos01} and \cite[Remark 6.1.9]{Huy}) and the assertion follows from our assumption.

\end{proof}

Slope stable elements can be readily obtained. 
Let $r_0$ be the minimal rank of $\srX$-twisted sheaves in $\mathrm{Coh}^{(1)}(\srX)$. Then $r_0$ equals the order of $[\srX]$ in $\Br(X)$, and every element in $\mathrm{Coh}^{(1)}(\srX)$ has rank $a r_0$ for some integer $a \in \ZZ$ (cf.~\cite[Lemma 3.2]{yoshioka2006moduli}). 
For any primitive vector 
\[
\bfv = (a r_0, h, s) \in \widetilde{\rH}(\srX)
\] 
with $a > 0$ and $\gcd(a, H \cdot h) = 1$, the sheaves $\cE \in M_H(\srX, \bfv)$ are $\mu_H$-stable. When $\mathrm{rank}(\bfv) > \frac{\bfv^2 + 2}{2}$, we know that every element in $ M_H(\srX, \bfv)$ is locally free.

To ensure the existence of $\bfv$ with isotropic vector in its orthogonal complement, set $s = 0$. Then $(1,0,0) \in \bfv^\perp$ is isotropic.

\begin{remark}
    There is a more natural perspective on the construction of $\fro_\srX$: it can be obtained via the second Chern class of the Beauville-Voisin class on a Fourier-Mukai partner. This perspective applies specifically when $N = M_H(\srX, \bfw)$ is a K3 surface, in which case the local freeness requirement for $\cF$ (lying on a constant cycle Lagrangian subvariety) becomes unnecessary.
    
    Let $\cF \in N$ be any point on a constant cycle curve $C \subseteq N$. In this case, $\bS_0^{\mathrm{BV}}\CH_0(N)$ is known to be $1$-dimensional and hence equal to $\bS_0^{\mathrm{SYZ}}\CH_0(N)$. By Theorem \ref{mainthm}, for any $\cF \in C \subset N$, we have:
    \[
    c_2(\cF) - \rank(\cF)\fro_\srX = (\deg(c_2(\cF)) - \rank(\cF))\fro_X.
    \]
    This implies:
    \begin{equation}\label{eq:def-twistbv}
        \fro_\srX = \frac{c_2(\cF)}{\rank \cF} + \left(1 - \frac{\deg(c_2(\cF))}{\rank \cF}\right) \fro_X
    \end{equation}
    provided $\mathrm{rank}(\bfw) > 0$. For $\dim N > 2$, it remains open whether  the equation \eqref{eq:def-twistbv} holds for arbitrary $\cF$ lying on a constant cycle Lagrangian subvariety.
\end{remark}

\begin{remark}[Descent property]\label{rmk:descent}
  We speculate that  $\fro_\srX=\fro_X$. The main evidence is that  due to our construction, the twisted BV class $\fro_\srX$ has nice descent property. This follows from the descent property of constant cycle curves on K3 surfaces (\cf.~\cite[Proposition 3.6]{Huy14}).
  
 To be precise, suppose $X$ is defined over $\bar{k}$. Take a Mukai vector $\bfv=(r,h,s)$ such that $\bfv^2=0$ and $r> 0$. Then $M=M_H(\srX,\bfv)$ is a K3 surface defined over $\bar{k}$ and it consists of Gieseker stable $\srX$-twisted locally free sheaves.    Take a $\bar{k}$-rational curve  $C\subseteq M$ and let $\cF$ be a point on $C$. Then $\fro_{\srX}=\frac{c_2(\cF)}{r}+(1-\frac{\deg(c_2(\cF))}{r})\fro_X$ is clearly defined over $\bar{k}$. 
  
  When $\bar{k}=\bar{\QQ}$,  the Bloch-Beilinson conjecture predicts  $\CH_0(X_{\bar{\QQ}} )=\ZZ\fro_X$ and hence $\fro_\srX=\fro_X$.  
\end{remark}

In the rest, we will show that $\fro_{\srX}$ is independent of the choice of the moduli space $N$ and the constant cycle Lagrangian subvariety $Z$. The key result is 

\begin{proposition}\label{prop:main}
Let $\srM_i:=\srM_{\sigma_i}(\srX,\bfv_i),~i=1,\ldots,k $ and $\srN_j:=\srM_{\tau_j}(\srX,\bfw_j),~j=1,\ldots, \ell$  be  Bridgeland moduli stacks. Set $M_i$ (resp. $N_j$) to be the corresponding coarse moduli space.  Assume that  
\begin{enumerate}
\item [(i)] $M_i^{\rm l.f.}$ and $N_j^{\rm l.f.}$ are non-empty;
\item [(ii)] $\cE_{i}^\vee \otimes \cF_{j}$ is globally generated and
     \begin{equation}\label{eq:vanishing1}
         \rH^{n}(X,\cE_{i}^\vee \otimes \cF_{j})=0,\quad \forall ~n=1,2,
     \end{equation}
     for any $\cE_i\in M_i$ and $\cF_j\in N_j $;
\item [(iii)] for $\cG \in M_i$ (or~$N_i$) and $\cG'\in M_{i'}$ (resp.~$N_{i'}$),  we have 
    \begin{equation}\label{eq:vanishing2}
        \rH^1(\cG^\vee\otimes \cG')=0,\quad\forall~ i\neq i'
    \end{equation}and 
    \begin{equation}\label{eq:vanishing3}
        \rH^0(\cG^\vee\otimes \cG')=0 ,\quad\forall~ i> i'
    \end{equation}
\end{enumerate}
Let $\cE=\bigoplus\limits_{i=1}^k\cE_i$ and $\cF=\bigoplus\limits_{j=1}^\ell \cF_j$ with $\cE_i\in M_i^{\rm l.f.} $ and $\cF_j\in N_j^{\rm l.f.}$. If $\rank \cE=\rank \cF$, then
    \begin{equation}
     c_2(\cE)-c_2(\cF)\in \bS_{d}(X). 
    \end{equation}
    where $d=\sum\limits_{i=1}^k (\rd(\cE_i)-n_i)+\sum\limits_{j=1}^\ell (\rd(\cF_j)-m_j)$, $\dim O_{\cE_i}\cap M_i^{\rm l.f.} =n_i$ and $\dim O_{\cF_j}\cap N_j^{\rm l.f.} =m_j$.

\end{proposition}

The proof will be given later.   With this result, we have 
\begin{corollary}\label{cor:well-def}
     $\fro_{\srX}$ is independent of the choice of $N$ and $\cF$. 
\end{corollary}
\begin{proof}
First, we fix $\bfv=(r,L+rB,0)$ a primitive Mukai vector which satisfies the assumptions in Proposition \ref{prop:existence}.  Let $\cF$ be a point lying on a  constant cycle Lagrangian subvariety  $Z\subseteq N$.

Suppose $N'$ is another moduli space of twisted locally free sheaves and it contains a constant cycle Lagrangian subvariety $Z'$. Let $\cE\in N'$. We set $\ell=\rank (\cF)$ and $k=\rank(\cE)$. Then one can take an ample line bundle $H$ and define 
\begin{itemize}
     \item  $\cE_i=\cE\otimes H^{\otimes a_i},~i=1,\ldots,\ell$.
    \item $\cF_j=\cF\otimes H^{\otimes b_j},~j=1,\ldots, k$;
\end{itemize}  
By choosing $b_j-b_{j-1}$, $a_i-a_{i-1}$ and $b_1-a_k$ sufficiently large, $\cF_i$ and $\cE_i$ satisfy the conditions in Proposition \ref{prop:main} by Lemma \ref{lem:langer} below. Due to Remark \ref{rmk:tensor}, this does not change the associated twisted BV class. Thus  we have \begin{equation}\label{well-defined}
   \begin{aligned}
         c_2(\bigoplus\limits_{j=1}^k \cF_i)-c_2(\bigoplus\limits_{i=1}^\ell \cE_i) \in \bS_0(X)\\
   \end{aligned} 
\end{equation}
which is equivalent to $ k c_2(\cF)-\ell c_2(\cE) \in \bS_0(X)$.
\end{proof}

\begin{lemma}\label{lem:langer}
Let $M$ be a  scheme of finite type. Let $\srE$ be a coherent sheaf on $M\times X$ and let $H$ be an ample line bundle on $X$.  
Then there exists $d \in\NN$ such that whenever $k \geq d$, $\rH^i(\scE_b\otimes H^{\otimes k})=0, i=1,2$  for any $b\in M $.
\end{lemma}
\begin{proof}
This can be done by induction on the dimension of $M$. 
When $\dim M=0$, this is clear. Suppose this holds when $\dim M\leq n$. If $\dim M=n+1$,  let $N^i$ be the support of  the sheaf
$$\bR^i f_*(\srE\otimes g^*H^{\otimes k}),~i=1,2 $$
which is a closed subscheme of $M$ of finite type. Here,  $f:M\times X\to M$ and $g:M\times X\to X$ are projections. By Serre vanishing,  there exists $d_0\in \NN$ such that  $\dim N^i<\dim M$ when $k>d_0$.   By induction hypothesis, there exists $d_1\in\NN$ such that $\rH^i(\cE_b\otimes H^{\otimes k})=0$ for $b\in N_i$. Take $d=\max\{d_0,d_1\}$.
\end{proof}

\subsection{Proof of Proposition \ref{prop:main}}
The idea is to consider the degeneracy loci of morphisms between twisted sheaves. Let us recall the  definition and some basic properties.  
\begin{definition}\label{def:deg}
    Let $\cE$ and $\cF$ be two $\srX$-twisted locally free sheaves.
    For $t:\cE\to \cF\in \Hom(\cE,\cF)$, we define the $k$-th degeneracy locus $\srD_k(t)$ to be the closed substack of $\srX$ 
    defined by the image of the map $$\wedge^{k+1}t^\vee: \srH om(\wedge^{k+1}\cE,\wedge^{k+1}\cF)^\vee\to \cO_\srX.$$  We can also define the schematic $k$-th degeneracy locus $D_k(t)$ which is the scheme-theoretic image in $X$. (\cf.~\cite[Definition 2.2.6.2]{Lie07})
\end{definition}

Form the definition, since $\srH om(\wedge^{k+1}\cE,\wedge^{k+1}\cF)$ is the pullback of a sheaf on $X$, $\srD_k(t)\to D_k(t)$ is the pullback of   $\srX\to X$ via $D_k(t)\to X$. The following result is well-known for untwisted locally free sheaves (\cf.~\cite[\S 4.1]{B91}\cite[Lemma 2.14]{JL22}). 
\begin{lemma}\label{lemma:degenerte}
Let $\cE$ and $\cF$ be two $\srX$-twisted locally free sheaves of rank $e$ and $f$.
  Suppose $\cE^\vee\otimes \cF$ is globally generated. Then  for  general  $t\in \Hom(\cE,\cF)=\rH^0(\srX,\cE^\vee\otimes \cF)$,  $D_{k}(t)$ is either empty or has pure codimension $(e-k)(f-k)$. Moreover, the singular locus ${\rm Sing~}D_{k}(t)\subset D_{k-1}(t)$.
\end{lemma}
\begin{proof}

When $\srX\to X$ is trivial, this is exactly \cite[Lemma 2.14]{JL22}. In the case $\srX\to X$ is essentially trivial,  after taking  a twisted line bundle $\cL\in\Pic(\srX)$, we can identify $\Hom_{\cO_\srX}(\cE,\cF)$ with $\Hom_{\cO_X}(\cE\otimes \cL^\vee,\cF\otimes \cL^\vee)$ and their schematic degeneracy loci are the same. The assertion follows from the untwisted version.

In general,   we let $Y\to X$ be a Severi-Brauer variety of $\srX\to X$  with the Cartesian diagram
\begin{equation*}
    \begin{tikzcd}
\srY \arrow[d, "\varpi"] \arrow[r, ""] & Y \arrow[d, "\pi"] \\
\srX \arrow[r, ""]                & X               
\end{tikzcd}
\end{equation*}
Then $\srY\to Y$ is essentially trivial. Set $\cE'=\varpi^\ast\cE$ and $\cF'=\varpi^\ast \cF$. Then we have $$\pi^*(\cE^\vee \otimes \cF)\cong (\cE')^\vee\otimes \cF'$$ as untwisted sheaves. By projection formula, 
there is an isomorphism
\begin{equation}
    \pi^\ast:\rH^0(X, \cE^\vee \otimes \cF)\to \rH^0(Y,(\cE')^\vee\times \cF').
\end{equation} 

Since $\srY\to Y$ is essentially trivial, the assertion holds for $\pi^\ast t$ when $t$ is general. Note that $D_k(\pi^\ast t)=\pi^{-1}D_k(t)$ from the definition, we therefore get $$\codim D_k(t)=\codim D_k(\pi^\ast t)=(e-k)(f-k).$$
and $ {\rm Sing~} \pi^{-1}D_{k}(t)\subset \pi^{-1}(D_{k-1}(t))$.
As  $Y\to X$ is a $\PP^n$-fibration, one must have  ${\rm Sing}~\pi^{-1}D_{k}(t)=\pi^{-1}({\rm Sing} ~D_k(t))$. It follows that $D_k(t)\subseteq D_{k-1}(t)$ as well. 
\end{proof}

Let us come back to the proof of Proposition \ref{prop:main}. From the vanishing condition \eqref{eq:vanishing2} in assumption (iii),  we have $$\dim_\CC \Ext^1(\cE,\cE)=2\sum\limits_{i=1}^k \rd(\cE_i){\rm~ and~}\dim_\CC \Ext^1(\cF,\cF)=2\sum\limits_{j=1}^\ell \rd(\cF_j).$$

By assumption (ii), for general  $t\in \Hom(\cE,\cF),$ 
we have a short exact sequence
\begin{equation}\label{eq:exact}
0\to \cE\xrightarrow{t} \cF\to \cQ_t\to 0.
\end{equation}
According to Lemma \ref{lemma:degenerte}, $\cQ_t$ is supported on a smooth curve $ \srC_t$ and it is locally free of rank $1$ on its support.

Let $\mathscr{E}_{i}\in \rD^{(-1,1)}(\srX\times \srM_i)$ and $\mathscr{F}_{i}\in \rD^{(-1,1)}(\srX\times \srN_j)$  be the universal sheaf. Set $$\mathscr{G}:=\boxtimes_{i,j}(\srE_{i}^\vee \boxtimes \srF_{j})\in \mathrm{Coh}^{(-1,1,1)}(X\times \prod_{i=1}^{k} \srM_i \times \prod_{j=1}^{\ell} \srN_j).$$ 
For simplicity, we assume that $\srM_i=M_i$ and $\srN_j=N_j$ are  fine moduli spaces. Otherwise, one  may either work with everything over stacks directly or pass to the associated Severi-Brauer varieties just as in the proof of \cite[Theorem 3.1]{LiZ22}. Take an irreducible component $Y_i\subseteq O_{\cE_i}$ of dimension $n_i$ and $Z_j\subseteq O_{\cF_j}$ of dimension $m_j$,  we set $$Y=\prod_{i=1}^{k} Y_{i}~{\rm and}~ Z=\prod_{j=1}^{\ell}Z_{j}.$$
Then we can restrict $\pi_\ast(\scG)$ to $Y \times Z$ which is locally free by assumption (i).  Define 
    $\bP=\PP((\pi_\ast \scG)|_{Y\times Z})$ to be the associated projective bundle over $Y \times Z$.  There is a diagram 
\begin{equation*}
\begin{tikzcd}
\bP \ar[r,dotted,"\psi"] \ar[d, "\pi"] & M_H(\srX,\bfv_0) \\
Y\times Z & 
\end{tikzcd} 
\end{equation*}
where the rational map $\psi$ is defined by sending the arrow $\cE \xrightarrow{t} \cF$ to $\cQ_t$. Let $\bP_{y,z}=\pi^{-1}(y,z)$ be the fiber of $\pi$ over $(y,z)\in Y\times Z$. 

 \begin{claim} \label{claim:fiber}
 $\psi(\bP_{y_1,z_1})\cap \psi(\bP_{y_2,z_2})\neq \emptyset$ if and only if $(y_1,z_1)=(y_2,z_2)$. Moreover, the  fiber of $\psi$ over $\cQ_t$ has dimension  
 \begin{equation}\label{eq:dim-fiber}
 \begin{aligned}
     \rd(\cE)+\rd(\cF)-\frac{\langle \bfv(\cE),\bfv(\cE)\rangle}{2}-\frac{\langle \bfv(\cF),\bfv(\cF)\rangle}{2}-2
 \end{aligned} 
 \end{equation}
 for general $t$.
 \end{claim}

Using \eqref{eq:dim-fiber}, we can deduce that
\begin{equation*}
\begin{aligned}
     \dim \mathrm{Image}(\psi)&=\dim Y+\dim Z+h^0(\pi_\ast\cG|_{Y\times Z})-1-\dim \hbox{fiber of $\psi$}
     \\ 
     & =\sum\limits_{i=1}^{k} n_i+ \sum\limits_{j=1}^{\ell} m_j -\langle \bfv(\cE),\bfv(\cF)\rangle -( \rd(\cE)+\rd(\cF)-\frac{\langle \bfv(\cE),\bfv(\cE)\rangle}{2}-\frac{\langle \bfv(\cF),\bfv(\cF)\rangle}{2}-1)\\
     &=\sum\limits_{i=1}^{k} n_i+ \sum\limits_{j=1}^{\ell} m_j+\rd(\cQ_t)-( \rd(\cE)+\rd(\cF))\\ &=\rd(\cQ_t)-d.
     \end{aligned}
\end{equation*} 
 Due to Theorem \ref{thm:MZ}, $\mathrm{Image}(\psi)$ is a constant cycle subvariety.  By Theorem \ref{thm:torsion}, we get 
  \begin{equation}
     c_2(\cE)-c_2(\cF)=c_2(\cQ_t)\in \bS_{d}(X). 
    \end{equation}
 as desired.

It remains to prove Claim \ref{claim:fiber}.
 Suppose there is another short exact sequence
    \begin{equation}\label{eq:exact1}
0\to \cE'\xrightarrow{t'} \cF'\to \cQ_{t'}\to 0,
\end{equation}
with $\cF'=\bigoplus\limits_{j=1}^\ell \cF_j'$  and $\cQ_{t'}\cong \cQ_{t}$. After applying the functor $\Hom(\cF_i,-)$ to \eqref{eq:exact1} and \eqref{eq:exact}, we get
\begin{equation}
\begin{tikzcd}
{0=\Hom(\cF_i,\cE)} \arrow[r] \arrow[d, "\cong"] & {\Hom(\cF_i,\cF)} \arrow[r]  & {\Hom(\cF_i,\cQ_t)} \arrow[r] \arrow[d, "\cong"] & {\Ext^1(\cF_i,\cE)=0} \\
{0=\Hom(\cF_i,\cE)} \arrow[r]                    & {\Hom(\cF_i,\cF')} \arrow[r] & {\Hom(\cF_i,\cQ_t')} \arrow[r]                   & {\Ext^1(\cF_i,\cE)=0}
\end{tikzcd}
\end{equation}
Here the vanishing is guaranteed by \eqref{eq:vanishing1} and $\cF_i$ is a direct summand of $\cF$.
Therefore, we have 
$$\dim_\CC \Hom(\cF_i, \cF')=\dim_\CC \Hom(\cF_i,\cQ_t')=\dim_\CC \Hom(\cF_i,\cQ_{t'}')=\dim_\CC \Hom(\cF_i,\cF).$$ 
Note that by \eqref{eq:vanishing2} and \eqref{eq:vanishing3}, we have
$$\dim_\CC \Hom(\cF_i, \cF_j)=\dim_\CC \Hom(\cF_i, \cF_j')=-\langle\bfv_i,\bfv_j\rangle,\quad \forall j> i$$ and $$\dim_\CC \Hom(\cF_i, \cF_j)=\dim_\CC \Hom(\cF_i, \cF_j')=0,\quad \forall j< i.$$
Therefore, we obtain $$\dim_\CC \Hom(\cF_i,\cF_i')=\dim_\CC \Hom(\cF_i,\cF_i)=1.$$
This forces $\cF\cong \cF'$ since  $\cF_i$ and $\cF_i'$ are stable of the same phase. 
Similarly, if we apply 
the functor $\Hom(\cE_i,-)$, we can get 
\begin{equation}
    \dim_\CC \Hom(\cE_i',\cE_i) =\dim_\CC \Ext^2(\cE_i,\cE_i')=\dim_\CC \Ext^2(\cE_i,\cE_i)=1. 
\end{equation}
This implies $\cE\cong \cE'$ and proves the first claim.   The fiber $\psi^{-1}(\cQ_t)$ consists of the kernel of $\cF\to \cQ_t$ and we have 
\begin{equation*}
\begin{aligned}
 \dim \psi^{-1}(\cQ_t)&=\dim_\CC \Hom(\cF,\cQ_t)+\dim_\CC\Hom(\cE,\cE)-\dim_\CC \Hom(\cQ_t,\cQ_t)-1   \\  &=\dim_\CC \Hom(\cF,\cF)+\dim_\CC \Hom(\cE,\cE)-\dim_\CC \Hom(\cQ_t,\cQ_t)-1 \\ &=\rd(\cF)-\frac{\langle \bfv(\cF),\bfv(\cF)\rangle}{2}+\rd(\cE)-\frac{\langle \bfv(\cE),\bfv(\cE)\rangle}{2} -2.
\end{aligned}    
\end{equation*} 
Here, $\dim_\CC\Hom(\cF,\cQ_t)=\dim_\CC \Hom(\cF,\cF)$ is obtained by applying the functor $\Hom(\cF,-)$ to \eqref{eq:exact} and $\dim_\CC \Hom(\cQ_t,\cQ_t)=1$ as $\cQ_t$ is stable.
\qed

\subsection{Proof of Theorem \ref{mainthm}}  Let $\cE\in \rD^{(1)}(\srX)$.
As the argument proceeds a similar way as in \cite{Voi08,SYZ20}, we only sketch the proof.
\subsubsection*{i). $\cE$ is a sheaf}
  Set $r=\rank(\cE)$.  If $r=0$,  this is Theorem \ref{thm:torsion}.  When $r>0$, we have the following situations
  \begin{itemize}
  \item If $\cE$ is locally free and simple,  we can find two primitive Mukai vectors $\bfv', \bfw\in \widetilde{\rH}(\srX)$ satisfying the conditions  in Proposition \ref{prop:existence} and $$\rank (\bfv')+\rank(\cE)=\rank(\bfw).$$   By Proposition \ref{prop:existence}, there exist $\srX$-twisted locally free sheaves 
  \begin{center}
     $\cE' \in M_{H}(\srX,\bfv')$ and $\cF\in M_{H}(\srX,\bfw)$    
  \end{center}
 lying on some constant cycle Lagrangian subvarieties respectively. In other words, $\dim O_{\cE'}=\rd(\cE')$ and $\dim O_\cF=\rd(\cF)$.

As in the proof of Corollary \ref{cor:well-def},  up to tensoring sufficiently ample line bundles,  we assume  $\cE$, $M_{H}(\srX,\bfv')$  and $M_{H}(\srX,\bfw)$ satisfy the conditions in Proposition \ref{prop:main}.   It follows from Proposition \ref{prop:main} that   $$c_2(\cE)+c_2(\cE')-c_2(\cF)\in \bS_{\rd(\cE)}(X).$$
This implies $c_2(\cE)-r\fro_{\srX}\in \bS_{\rd(\cE)}(X)$ from the definition of $\fro_\srX$. 

\item If $\cE$ is only torsion-free, the double dual $\cE^{\vee\vee}$ of $\cE$ is locally free. There is a short exact sequence of sheaves
$$0 \to \cE \to  \cE^{\vee\vee} \to \cQ \to  0$$
where $\cQ$ is a 0-dimensional sheaf whose support is of length $\ell$. A direct calculation yields $\mathrm{d}(\cQ) > \ell $ and $c_2(\cQ) \in \bS_\ell(X) \subseteq \bS_{\mathrm{d}(\cQ)}(X)$.
 When $\cE$ is slope stable, the double dual $\cE^{\vee\vee}$ is also stable and hence simple. This gives
$$c_2(\cE^{\vee\vee})-r\fro_{\srX}\in \bS_{\mathrm{d}(\cE^{\vee\vee})}(X).$$
Consider the distinguished triangle $\cQ[-1] \to \cE \to \cE^{\vee\vee}\to \cQ, $  we have
$$\Hom(\cQ[-1], \cE^{\vee\vee}) = 0,$$
as  $\cQ$ is $0$-dimensional and $\cE^{\vee\vee}$ is locally free. Then $\rd(\cE) \leq \rd(\cE^{\vee\vee})+\rd(\cQ[-1])$. It follows that $$c_2(\cE)-r\fro_{\srX} =c_2(\cE^{\vee\vee})-r \fro_{\srX}-c_2(\cQ)\in \bS_{\rd(\cE^{\vee})}(X)+\bS_{\rd(\cQ)}(X).$$ 
When $\cE$ is not slope stable, one can obtain the assertion by applying the argument above to the  HN factors of $\cE$. 
\item If $\cE$ is not torsion free, there is a short exact sequence 
\begin{equation}
    0\to \cT\to \cE\to \cF\to 0
\end{equation}
with $\cT$ torsion and $\cF$ torsion-free. Then we have 
\begin{equation}
    \begin{aligned}
        c_2(\cE)-r\fro_{\srX}&=c_2(\cT)+(c_2(\cF)-r\fro_{\srX})\\
        &\in \bS_{\rd(\cT)+\rd(\cF)}(X) \subseteq \bS_{\rd(\cE)}(X).
    \end{aligned}
\end{equation}
The last inequality holds as $\Hom(\cT, \cF) = 0$ and hence  $\rd(\cT)+\rd(\cF)\leq \rd(\cE)$. 

  \end{itemize}

\subsubsection*{ii). $\cE$ is not a sheaf}   There is a standard distinguished triangle
$$\cF \to  \cE \to  \cG \to  \cF[1].$$
where $\cG$ is the shifted sheaf $h^m(\cE)[-m]$  and $\cF\in\rD^{(1)}(\srX)$ is the truncated complex $\tau^{\leq m-1}$ 
By induction on the length of $\cE$, one can get $c_2(\cE)-\rank(\cE)\fro_{\srX}\in \bS_{\rd(\cE)}(X)$. 
\qed

\section{Proof of Main Theorems} \label{sec:proof of mainthm2}

\subsection{Beauville-Voisin subgroup}
Let $\rR(X)\subset \CH(X)$ be the Beauville-Voisin subring. 
Given a derived equivalence $\Psi:\rD^{(1)}(\srX)\to \rD^{(1)}(\srY)$, it induces a (modified) additive group isomorphism 
$$\begin{aligned}
    \widetilde{\Psi}^{\CH}:\CH^*(X) &\to \CH^*(Y),\\
    \widetilde{\ch}_\srX(\cE) & \mapsto \widetilde{\ch}_{\srY}(\Psi(\cE))Y
\end{aligned}$$
where $\widetilde{\ch}_\srX(\cE)=\ch_\srX(\cE)-\rank(\cE)\fro_{\srX}$ and $\widetilde{\ch}_\srY(\Psi(\cE))=\ch_\srY(\Psi(\cE))-\rank(\Psi(\cE))\fro_{\srY}.$ Then the following result can be viewed  as a generalization of \cite{Huy10}. 
\begin{proposition}\label{prop:BV}
The group isomorphism $\widetilde{\Psi}^{\CH}$ preserves the Beauville-Voisin group.
\end{proposition}
\begin{proof}
Let $\bfv\in \widetilde{\rH}(\srX)$ be a primitive Mukai vector whose orthogonal complement in $\widetilde{\rH}_{\rm alg}(\srX)$ contains an isotropic element. Note that $\rR(X)$ is spanned by such vectors. It suffices to show 
for some $\cE\in M_\sigma(\srX,\bfv)$ with $\widetilde{\ch}_\srX(\cE)\in \rR(X),$ we have 
\begin{equation}\label{eq:FM-BV}
    \widetilde{\ch}_\srY(\Psi(\cE))\in \rR(Y).
\end{equation}

Let $\dim M_{\sigma}(\srX,v)=2n$. 
Note that $M_\sigma(\srX,\bfv)$ admits a birational Lagrangian fibration. As proved in Proposition \ref{prop:dense-lag}, there exists $\cE\in M_{\sigma(\srX,\bfv)}$ such that $O_\cE$ has dimension $n$ and its maximal dimensional irreducible components are Zariski dense in $M_{\sigma}(\srX,\bfv)$.  
Set $\cF=\Psi(\cE)\in \rD^{(1)}(\srY)$ and $\bfw=\bfv(\cF)$.   The derived equivalence $\Psi$ induces an isomorphism
\begin{equation}\label{eq:de-iso}
\begin{aligned}
      M_\sigma(\srX, \bfv)&\to M_\tau(\srY,\bfw)\\
      x&\mapsto \Psi(x)
\end{aligned}
\end{equation}
with $\tau=\Psi(\sigma)$.  Then $\dim O_\cF=n$ and its maximal dimensional irreducible components are also Zariski dense in $M_{\tau}(\srY,\bfw)$.   Using the twisted incidence variety $\Gamma\subseteq M_\tau(\srY,\bfw)\times Y^{[n]}$ defined above, one can immediately get $c_2(\cF)-\rank(\cF)\fro_\srY\in \bS_0(Y)$
as the projections $\Gamma\to Y^{[n]}$ and $\Gamma\to M_\tau(\srY,\bfw)$ are generically finite and dominant. 
 \end{proof}

\begin{remark}
   In \cite[Question 4.1]{SYZ20}, it has been asked if the map
$$\begin{aligned}
    \Psi^{\CH}:\CH^*(X) &\to \CH^*(Y),\\
    \ch_\srX(\cE) & \mapsto \ch_{\srY}(\Psi(\cE))
\end{aligned}$$
preserves  the Beauville-Voisin group. By Proposition \ref{prop:BV}, this question is equivalent to ask whether $\fro_{\srX}=\fro_X$. 
\end{remark}

\subsection{Proof of Theorem \ref{mainthm2}}
Suppose we have $$c_2(\cE)-\rank(\cE)\fro_{\srX}=x_1+\ldots +x_i+k\fro_\srX,$$ where $x_1,\ldots ,x_i$ are  distinct points on $X$ and $k$ is a constant. Take the direct sum  $$ \cE'\cong \CC_{x_1}\oplus \cdots\oplus \CC_{x_i}\in \rD^{(1)}(\srX).$$
Then $\widetilde{\ch}(\cE)-\widetilde{\ch}(\cE')\in \rR(X)$.
Since $\widetilde{\Psi}^{\CH}$ preserves the Beauville-Voisin group, we get 
\begin{equation}
    \widetilde{\ch}(\Psi(\cE))-\widetilde{\ch}(\Psi(\cE'))\in \rR(Y). 
\end{equation}
By Theorem \ref{mainthm},  $\Psi(\cE')\in \bS_{\rd(\cE')}(\rD^{(1)}(\srY))=\bS_{i}(\rD^{(1)}(\srY))$. Hence $\Psi(\cE)\in \bS_i(\rD^{(1)}(\srY))$ as well. \qed

\subsection{SYZ filtration v.s. Voisin's filtration}

We now generalize the work in \cite{SYZ20} and \cite{LiZ22} to twisted K3 surfaces. Assume that $M_\sigma(\srX,\bfv)$ is a moduli space of $\sigma$-stable locally free sheaves on $\srX$ and $\sigma$ is $\bfv$-generic. Using the twisted Beauville-Voisin class $\fro_{\srX}$, we define the twisted incidence correspondence:
\begin{equation}\label{eq:incident}
    \Gamma' := \left\{ (\xi, \cE) : c_2(\cE) - \rank(\cE) \fro_{\srX} = [\xi] + k \fro_X \in \CH_0(X) \right\} \subset X^{[d]} \times M_\sigma(\srX,\bfv),
\end{equation}
for an arbitrary Mukai vector $\bfv$. By Theorem \ref{mainthm}, the projection maps $\pi_1 \colon \Gamma' \to M_\sigma(\srX, \bfv)$ and $\pi_2 \colon \Gamma' \to X^{[d]}$ are both dominant.

Thus, there exists a subvariety $\Gamma \subseteq \Gamma'$ that is generically finite over both $M_\sigma(\srX,\bfv)$ and $X^{[d]}$. For two general points $p, q \in \Gamma$ lying on the same fiber of $\pi_1$ (respectively $\pi_2$), Theorem \ref{thm:MZ} implies that $\pi_2(p) = \pi_2(q) \in \CH_0(X^{[d]})$ (respectively $\pi_1(p) = \pi_1(q) \in \CH_0(M_\sigma(\srX,\bfv))$). Consequently, if $Z \subset X^{[d]}$ is a constant cycle subvariety, then $\pi_1(\pi_2^{-1}(Z))$ is also constant cycle. Using the construction above, we obtain

\begin{theorem}\label{thm:SYZimpliesBV}
 If $c_2(\cE)-\rank(\cE)\fro_{\srY}\in \bS_i(X)$, then $\dim O_\cE\geq d-i.$  Conversely,  if $\dim O_\cE \geq d-i$, then $\cE\in \bS_i(\rD^{(1)}(\srX))$.  Consequently, we have
    \[
    \bS^{\mathrm{SYZ}}_\bullet \CH_0(M_\sigma(\srX,\bfv)) = \bS^{\mathrm{BV}}_\bullet \CH_0(M_\sigma(\srX,\bfv)).
    \]
\end{theorem}

\begin{proof}
 By Theorems \ref{mainthm2} and \ref{thm:locally-free}, we may assume (after applying a derived equivalence) that $M_\sigma(\srX,\bfv)$ consists of locally free twisted sheaves. The remainder of the proof is identical to \cite[Theorem 0.5]{SYZ20}.
\end{proof}

As an application, we  obtain a Zariski density result for constant cycle subvarieties. 
\begin{corollary}\label{cor:dense}
For any $\cE\in M_{\sigma}(\srX,\bfv)$ with dim $\dim O_\cE= d-i$, the subvarieties of the orbit $O_\cE$ with dimension $d-i$ are dense in $M_{\sigma}(\srX,\bfv)$.
\end{corollary}
\begin{proof}
  Again, we may assume $M_\sigma(\srX,\bfv)$ consists of locally free twisted sheaves. Then this can be obtained along  the same lines as in  the proof \cite[Corollary 3.2]{LiZ22}.
\end{proof}

\subsection{Proof of Corollary \ref{cor：bir-invariant}}
By \cite[Proposition 4]{add16}, \cite[Lemma 2.6]{Huy17}, and \cite[Corollary 3.3]{LYZ23}, the condition on the Markman-Mukai lattice implies that $Y \cong M_\sigma(\mathscr{X}, \mathbf{v})$ for some twisted K3 surface $\mathscr{X} \to X$. Theorem \ref{mainthm3} establishes:
\[
\mathbf{S}^{\mathrm{BV}}_\bullet \operatorname{CH}_0(Y) = \mathbf{S}^{\mathrm{SYZ}}_\bullet \operatorname{CH}_0(Y).
\] 
 To prove birational invariance, we adapt the argument in \cite[Corollary 3.3]{LiZ22} to the twisted case. Since the original proof omits some technical details, we provide a comprehensive exposition below.

By \cite[Theorem 1.2]{BM14}, the hyper-Kähler birational models of $Y$ correspond to chambers in $\operatorname{Stab}^\dagger(\mathscr{X})$.
As shown in \cite[Corollary 3.3]{LiZ22}, any birational model $Y'$ is isomorphic to $M_{\sigma'}(\mathscr{X}, \mathbf{v})$ for some generic $\sigma'$ in a chamber $\mathcal{C}'\subseteq \operatorname{Stab}^\dagger(\mathscr{X})$, and the birational map 
\[
f: M_\sigma(\mathscr{X}, \mathbf{v}) \dashrightarrow M_{\sigma'}(\mathscr{X}, \mathbf{v})
\] 
is induced by a derived (anti-)autoequivalence $\Phi \in \operatorname{Aut}(\mathrm{D}^{(1)}(\mathscr{X}))$ (up to automorphism).

Let $U \subset Y$ be the domain where $f$ is defined. For any $E \in U$, Theorem \ref{mainthm2} guarantees:
\[
E \in \bS_i^{\rm SYZ}(\rD^{(1)}(\mathscr{X})) \implies \Phi(E) \in \bS_i^{\rm SYZ}(\rD^{(1)}(\mathscr{X}))
\]F
since derived equivalences preserve the SYZ filtration. Thus $f(E) \in \bS_i^{\rm SYZ}\CH_0(Y')$. Now consider any  point $F$ with $[F] \in \bS_i^{\SYZ}\CH_0(Y)$. By definition of the filtration, we have 
\begin{align}\label{eq:chernid}
[F] = \sum n_j [F_j] \quad \text{with} \quad c_2(F_j) \in \bS_i(\srX).
\end{align}
By Theorem \ref{thm:MZ}, the class $[F_j]$ is determined by $c_2(F_j)$. 
Using the surjective projections $\pi_1 : \Gamma' \to M_\sigma(\mathscr{X}, \mathbf{v})$ and $\pi_2 : \Gamma' \to X^{[d]}$ from the twisted incidence correspondence  \eqref{eq:incident}, we may adjust the $F_j$ to lie in $U$   while preserving \eqref{eq:chernid}. 

Applying $f_*$ we obtain:
\[
f_*[F] = \sum n_j [f(F_j)] \quad \text{with} \quad f(F_j) \in \bS_i^{\SYZ}\CH_0(Y').
\]
Hence $f_*[F] \in \bS_i^{\SYZ}\CH_0(Y')$. The reverse inclusion follows by applying the same argument to $f^{-1}$. We conclude:
\[
f_*(\bS_\bullet^{\SYZ}\CH_0(Y)) = \bS_\bullet^{\SYZ}\CH_0(Y').
\]
Theorem \ref{mainthm3} establishes $\mathbf{S}_\bullet^{\mathrm{BV}} = \mathbf{S}_\bullet^{\mathrm{SYZ}}$, completing the proof of birational invariance.

For the last assertion, note that if $\rank \Pic(Y) \notin \{2,3\}$, then either the algebraic Markman-Mukai lattice of $Y$ (which has rank $\geq 5$) contains an isotropic vector, or $Y$ admits no nontrivial birational hyper-K\"ahler model. The assertion follows immediately.  \qed

\subsection{Further remarks}\label{rmk:kummer}
There is a natural (conjectural) O'Grady filtration for abelian surfaces, which induces a Shen-Yin-Zhao (SYZ) type  filtration on albanese fiber of Bridgeland moduli space for (twisted) abelian surfaces. More precisely,  if $A$ is projective abelian surface, an O'Grady type set-theoretic filtration is given by 
\begin{equation}
    \bS_i(A):=\bS_i(\mathrm{Km}(A))\subseteq \CH_0(A)
\end{equation}
Here, $\CH_0(\mathrm{Km}(A))$ can be identified as a subset of $\CH_0(A)$ via the pullback of the rational map $A\dashrightarrow \mathrm{Km}(A)$. If $M=M_\sigma(\mathscr{A},\bfv)$ is a smooth projective  Bridgeland moduli space  of $\sigma$-stable objects on a twisted abelian surface $\mathscr{A}\to A$, the kernel $K:=K_\sigma(\mathscr{A},\bfv)$ of the albanese map 
\begin{equation}
\begin{aligned}
    \mathfrak{alb}\colon M_{\sigma}(\mathscr{A},\bfv)&\longrightarrow \widehat{A}\times A,\\
  \cF &\mapsto (\det(\cF)\otimes \det(\cF_0)^{-1}, \mathrm{alb}(c_2(\cF)-c_2(\cF_0)),
\end{aligned}
\end{equation}is a smooth projective hyper-K\"ahler variety, where $\cF_0\in M_\sigma(\mathscr{A},\bfv)$. Note that for each $\cF\in K$, one has 
$$c_2(\cF)-c_2(\cF_0)\in \CH_0(\mathrm{Km}(A))_{\hom}=\CH_0(A)_{\rm alb}.$$

If there is a $\cF_0$ with $c_2(\cF_0)\in \bS_0(A)$,  one can  thus define a Shen-Yin-Zhao type filtration on $K_\sigma(\mathscr{A},\bfv)$. In this situation, our construction of (twisted) incidence correspondences is also valid and one can obtain  a similar result for the degeneracy loci of stable locally free sheaves on twisted abelian surfaces. The main problem will be the very first step, that is to check  whether the induced SYZ-type filtration  coincides with Voisin's filtration when $K=\ker(A^{[n+1]}\to A)$. Note that in this case, one does have an element $\cF_0\in K$ with $c_2(\cF_0)\in \bS_0(A)$.

\section{Applications to Bloch's conjecture}\label{sec:bloch}
In  this section, we consider the action of (anti)-symplectic birational automorphism on the Chow group of a Bridgeland moduli space.  

\subsection{Bloch's conjecture for zero cycles}
For a smooth projective hyper-K\"ahler variety $Y$,  Beauville and Voisin predicted that there is  an increasing filtration $\bS_\bullet\CH^\ast(Y)$  which serves as an opposite filtration to the famous conjectural Bloch-Beilinson filtration. For zero cycles, it is expected 
\begin{equation}
\bS_\bullet\CH_0(Y)=\bS_\bullet^{\rm BV}\CH_0(Y). 
\end{equation}

A more approachable question to examine these conjectures (and expectations) is to consider the action of (anti)-symplectic birational automorphisms on $\mathrm{Gr}_\bullet \CH_0(Y)$.  More precisely, we have the following conjecture

\begin{conjecture}[Bloch's conjecture for (anti)-symplectic birational map]\label{conj}
Let $Y$ be a hyper-K\"ahler variety of dimension $2n$.  Set $$\Gr_{s}\CH_{0}(Y):=\bS_{s}^{\rm  BV}\CH_0(Y)/\bS_{s-1}^{\rm BV}\CH_0(Y)$$ to be the $s$-th graded piece. For any $\phi\in \mathrm{Bir}(Y)$, we have   
\begin{equation*}
    \phi^\ast|_{\rH^{2,0}(Y)}=\pm \id \Leftrightarrow  \phi_\ast|_{\mathrm{Gr}_s\CH_0(Y)}=(\pm 1)^s\id,~\forall ~1\leq s\leq n.
\end{equation*}
\end{conjecture}

This conjecture has been studied in \cite{LYZ23} for untwisted Bridgeland moduli spaces, i.e. when $Y$ is isomorphic to the moduli space of stable objects in $\rD^b(X)$ for some K3 surface $X$.  The method in \cite{LYZ23} is also valid for twisted Bridgeland moduli spaces.  We will explain this in the next subsection. 

\subsection{Birational motives} 
We denote by $\frh^\circ(Y)$ the birational motive of $Y$ in the sense of \cite{KS16}. It has been conjectured in \cite{vial22} that there is a natural grading on $\frh^{\circ}(Y)$ compatible with  the co-algebra structure on $\frh^0(Y)$ (\cf.~\cite[Conjecture 2]{vial22}).  When $Y=M_\sigma(X,\bfv)$ for untwisted K3, such grading is governed by the isomorphism  $$\frh^\circ(Y)\cong \frh^\circ(Y)_{(0)}\oplus\ldots \oplus \frh^\circ(Y)_{(n)}$$
as co-algebra objects by \cite[Theorem 4]{vial22}, where $\dim Y=2n$. The grading on $\frh^\circ(Y)$ is compatible with $\bS_\bullet^{\rm SYZ}\CH_0(Y)$, i.e.
$$\bS_i^{\rm SYZ}\CH_0(Y)=\CH_0(\frh^\circ(Y)_{(0)}\oplus\ldots \oplus \frh^\circ(Y)_{(i)}),$$(\cf.~\cite[Section 6.4 and Theorem 7.3]{vial22}). 
Due to Theorem \ref{mainthm}, we can provide new examples for \cite[Conjecture 2]{vial22}.  The idea is that the Shen–Yin–Zhao type filtration admits a “motivic” splitting on $\frh^\circ(Y)$. 

\begin{proposition}\label{prop:birationalmotive}
If $Y=M_\sigma(\srX,\bfv)$, then the twisted incidence variety $\Gamma$ induces an isomorphism $$\frh^\circ(M_\sigma(\srX,\bfv))\cong \frh^\circ(X^{[n]}),$$ as co-algebra objects.
Moreover, the isomorphism respects the natural gradings. 
\end{proposition}
\begin{proof}
 As shown before,  $\Gamma \subseteq Y\times X^{[n]}$ is a generically finite  and dominant correspondence. The assertion is then a direct consequence of \cite[Proposition 2.3]{vial22}. 
\end{proof}

The key result is  
\begin{lemma}\label{lem:generator}
Let $Y=M_\sigma(\srX,\bfv)$ as before.   Let $Z\in \CH^{2n}(Y\times Y)_{\rm hom}$ be a correspondence. Then   we have $$Z_*|_{\mathrm{Gr}_s\CH_0(Y)} =(\pm)^s \id,$$ if $Z_*|_{\mathrm{Gr}_1\CH_0(Y)}=\pm \id$.
\end{lemma}
\begin{proof}
    By Proposition \ref{prop:birationalmotive}, we only need to treat the case $M_\sigma(\srX,\bfv)\cong X^{[n]}$. Then the result follows from \cite[Proposition 3.8(iii)]{LYZ23}.
\end{proof}

\subsection{Proof of Theorem \ref{thm:bloch}} 
As discussed in the proof Corollary \ref{cor：bir-invariant}, note that $Y$ admits a birational Lagrangian fibration, this implies:
\[
Y \cong M_\sigma(\srX,\bfv)
\]
for some twisted K3 surface $\srX \to X$, which endows $\CH_0(Y)$ with the Shen-Yin-Zhao filtration $\bS^{\mathrm{SYZ}}_\bullet \CH_0(Y) = \bS^{\mathrm{BV}}_\bullet \CH_0(Y)$. We prove both implications for $\phi \in \mathrm{Bir}(Y)$ preserving the birational Lagrangian fibration:
\begin{enumerate}
    \item[($\Rightarrow$)] When $\phi$ is (anti-)symplectic ($\phi^*|_{\rH^{2,0}(Y)} = \pm \id$), \cite[Theorem 2.17]{LYZ23} implies $\phi_*$ acts as $\pm \id$ on $\mathrm{Gr}_1\CH_0(Y)$. Lemma \ref{lem:generator} then yields the required action on higher graded pieces:
    \[
    \phi_*|_{\mathrm{Gr}_s\CH_0(Y)} = (\pm 1)^s \id \quad \forall 1 \leq s \leq n
    \]
    
    \item[($\Leftarrow$)] Suppose $\phi_*$ acts as $(\pm 1)^s \id$ on $\mathrm{Gr}_s\CH_0(Y)$ for all $1 \leq s \leq n$. For $s=1$, \cite[Lemma 2.1]{vial22} implies this action lifts to $\pm \id$ on $\frh^\circ(Y)_{(1)}$. Since $\frh^\circ(Y)_{(1)}$ has transcendental cohomology $\mathrm{T}(Y) \subset \rH^2(Y,\mathbb{Z})$ \cite[Proposition 10.24]{voisin2003hodge}, we obtain:
    \[
    \phi^*|_{\mathrm{T}(Y)} = \pm \id \implies \phi^*|_{\rH^{2,0}(Y)} = \pm \id. \qedhere
    \]
\end{enumerate}
This finishes the proof. \qed

\section{appendix}
In this appendix, we extend the result  in \cite{BM} to twisted K3 surfaces.  Let $\srX \to X$ be a twisted K3 surface, $\bfv^2 > 0 $ a Mukai vector, and $\sigma = (Z_{\sigma}, A_{\sigma}) \in \Stab(\srX)$ $\bfv$-generic. The first result is 

\begin{lemma}[cf.~{\cite[Lemma 7.3]{BM}}] \label{lem:BM}
Let $U$ be any connected open subset of $\Stab(\srX)$ containing $\sigma$ which does not encounter a wall. Then $U$ contains a dense subset of stability conditions $\tau = (Z_{\tau}, \srA_{\tau})$, for which there exists a primitive Mukai vector $\bfw$ with $\bfw^2=0$ such that $Z_{\tau}(w)$ and $Z_{\tau}(v)$ lie on the same ray in the complex plane.
\end{lemma}
\begin{proof}
The untwisted case is already treated in \cite[Lemma 7.3]{BM}. For twisted K3 surfaces, the argument is analogous and we provide details for the convenience of the reader. 

Assume $Z_{\sigma}(\mathbf{v}) = -1$ and restrict ourselves to those $\tau$ with $Z_{\tau}(\bfv) = -1$. Consider the quadric 
\[
\{ \mathbf{w}^2 = 0 \} \subset \widetilde{\mathrm{H}}(\mathscr{X}) \otimes \mathbb{R}.
\]
Since the Mukai pairing has signature $(4,20)$, there exists a real solution $\mathbf{w}_0$ satisfying $\mathrm{Im} Z_{\sigma}(\mathbf{w}_0) = 0$. As any nondegenerate rational quadratic form of dimension $\geq 5$ with indefinite signature represents $0$ over $\mathbb{Q}$, we can find $\mathbf{w}_1 \in \widetilde{\mathrm{H}}(\mathscr{X}) \otimes \mathbb{Q}$ arbitrarily close to $\mathbf{w}_0$, though $\mathrm{Im} Z_{\sigma}(\mathbf{w}_1) = 0$ may not hold.

Note that for $\mathbf{w}_1$ sufficiently close to $\mathbf{w}_0$, $\mathbf{w}_1$ and $\mathbf{v}$ are linearly independent. We can therefore deform $\sigma$ locally to obtain a new stability condition $\tau = (Z_{\tau}, A_{\tau})$ satisfying $\mathrm{Im} Z_{\tau}(\mathbf{v}) = \mathrm{Im} Z_{\tau}(\mathbf{w}_1) = 0$ while maintaining $\mathrm{Re}Z_{\tau} = \mathrm{Re}Z_{\sigma}$. Specifically, define the linear functional 
\[
\mathrm{Im} Z_{\tau} = \mathrm{Im} Z_{\sigma} + \delta
\] 
where $\delta$ satisfies
\begin{align*}
    \delta(\mathbf{v}) = 0, \quad \delta(\mathbf{w}_1) = -\mathrm{Im} Z_{\sigma}(\mathbf{w}_1),
\end{align*}
and vanishes on $\mathrm{span}(\mathbf{v}, \mathbf{w}_1)^{\perp}$. The norm $\|\delta\|$ is bounded by $|\mathrm{Im} Z_{\sigma}(\mathbf{w}_1)|$; since $\mathbf{w}_1$ is close to $\mathbf{w}_0$, this value is small, ensuring the Bridgeland distance $\|\sigma - \tau\|$ is small. Finally, replace $\mathbf{w}_1$ by the unique primitive integral class $\mathbf{w} \in \mathbb{R}\cdot \mathbf{w}_1$ satisfying $\mathrm{Re}Z_{\tau}(\mathbf{w}) < 0$.

\end{proof}

As a consequence, we get

\begin{theorem}\label{thm:locally-free}
There exists another twisted K3 surface $\srX'\to X'$ and a Fourier-Mukai equivalence
$\Phi: \rD^{(1)}(\srX) \longrightarrow \rD^{(1)}(\srX')$
which induces an isomorphism
\begin{align}
    \Phi \circ [-1]: M_{ \sigma}(\srX,\bfv) \xrightarrow{\sim} M_{\omega'}(\srX',-\Phi(\bfv))
\end{align}
from the Bridgeland moduli space of twisted $\sigma$-semistable objects of Mukai vector $\bfv$ on $\srX$ to the moduli space of twisted $\omega'$-Gieseker semistable bundles of Mukai vector $-\Phi(\bfv)$ on $\srX'$. Here $\omega'$ is a generic ample class on $\srX'$. Moreover,
\begin{enumerate}
    \item[i)] $(\omega', \beta')$-slope (semi)stability for twisted sheaves of class $-\Phi(\bfv)$ coincides with $\omega'$-Gieseker (semi)stability for any $\beta' \in \NS(\srX') \otimes \QQ$;
    \item[ii)] if $\bfv$ is primitive, both moduli spaces parametrize only stable objects.
\end{enumerate}
\end{theorem}

\begin{proof}
We divide the proof into three steps: Step 1 gives the construction of $\srX' \to X'$ and the equivalence $\Phi$; Step 2 clarifies that the image of objects with Mukai vector $\bfv$ under $\Phi$ are shifting of slope semistable twisted bundles; Step 3 relates the twisted slope stability with twisted Gieseker stability.
\subsubsection*{\underline{Step 1}}
Let $\tau$ be a generic element as in Lemma \ref{lem:BM}. 
According to \cite[Theorem 0.2]{HS}, there exists a coarse moduli space of stable $\alpha$-twisted sheaves with Mukai vector $\bfw$, such that the twisted universal sheaf $\srE$ induces an equivalence
\begin{align*}
\Phi: \rD^{(1)}(\srX) \longrightarrow \rD^{(1)}(\srX')
\end{align*}
where $X':= M_{\tau}(\srX,\bfw)$ and $\alpha' \in \Br(X')$ is the Brauer class of $\srX'\to X'$, making $\srE$ into a $\alpha^{-1} \boxtimes \alpha'$-twisted sheaf.

Now consider $\tau' := \Phi(\tau) \in \mathrm{Stab}(\srX')$. We may assume that $\tau' \in \Stab^{\dagger}(\srX')$ due to Bridgeland's covering map property. 
We now claim that up to acting by $\widetilde{\mathrm{GL}}_2^+(\RR)$, we may further assume that $\tau' = \sigma_{\omega', \beta'}$ for some $\omega', \beta' \in \NS(\srX')\otimes \QQ$, with $\omega'$ ample. As the construction of $\Phi$ indicates, $\Phi(\cF)$ is isomorphic to a skyscraper sheaf on $\srX' \to X'$ for every $\cF \in \srM_{\tau'}(\srX',\bfw)(\CC)$. Since $\Phi$ is an equivalence, every skyscraper sheaf is $\tau'$-stable with the same phase. The claim then follows from \cite[Remark 3.9]{HMS}.

\subsubsection*{\underline{Step 2} (Locally freeness)} 
We now show that for any $\cE \in \mathscr{M}_{\tau}(\mathscr{X},\mathbf{v})(\mathbb{C})$, $\Phi(\cE)[-1]$ is a twisted locally free sheaf on $\mathscr{X}' \to X'$ that is $\mu_{\omega', \beta'}$-semistable with slope $\mu_{\omega', \beta'}(\Phi(\cE)[-1]) = 0$.

Our choice of $\mathbf{w}$ ensures that $\mathbf{w}$ does not lie on a wall for $\mathbf{v}$; hence, no destabilizing factor of $\cE$ has Mukai vector $\mathbf{w}$. Consequently, after applying $\Phi$, no skyscraper sheaf is a stable factor of $\Phi(\cE)$. By Lemma \ref{lem:BM}, $Z_{\tau}(\mathbf{v})$ and $Z_{\tau}(\mathbf{w})$ lie on the same ray in the complex plane. Since
\begin{align*}
   Z_{\tau}(\mathbf{w}) &= Z_{\omega', \beta'}((0,0,1)_{X'}) \\
   &= \left< \exp(\beta'+i\omega'), (0,0,1)_{X'} \right> \in \mathbb{R}_{<0},
\end{align*}
$\Phi(\cE)$ has phase 1 and therefore lies in the heart $A_{\omega', \beta'}$. This gives a short exact sequence:
\begin{align*}
   0 \to \cH^{-1}(\Phi(\cE))[1] \to \Phi(\cE) \to \cH^{0}(\Phi(\cE)) \to 0,
\end{align*}
where $\cH^{\bullet}$ denotes the cohomology sheaves. Since $\mathrm{Im} Z_{\omega', \beta'}$ is additive and $\mathrm{Im} Z_{\omega', \beta'}(\cG) \geq 0$ for all $\cG \in \mathscr{A}(\omega', \beta')$, we have
\begin{align*}
    \mathrm{Im} Z_{\omega', \beta'}\left(\cH^{-1}(\Phi(\cE))\right) = \mathrm{Im} Z_{\omega', \beta'}\left(\cH^{0}(\Phi(\cE))\right) = 0.
\end{align*}

If $\cH^{0}(\Phi(\cE))$ has positive-dimensional support, then $\mathrm{Im} Z_{\omega', \beta'}\left(\cH^{0}(\Phi(\cE))\right) > 0$ by the Chern class definition of $Z_{\omega', \beta'}$. Hence, $\cH^{0}(\Phi(\cE))$ must be supported in dimension zero. Any stable factor of $\cH^{0}(\Phi(\cE))$ is then a stable factor of $\Phi(\cE)$, but no skyscraper sheaf can be a stable factor. Therefore, $\cH^{0}(\Phi(\cE)) = 0$, and $\Phi(\cE)[-1] \cong \cH^{-1}(\Phi(\cE))$. Similarly, $\cH^{-1}(\Phi(\cE))$ admits no nonzero maps to skyscraper sheaves, implying $\Phi(\cE)$ is a locally free twisted sheaf.

For slope semistability, note that if $\Phi(\cE)[-1]$ were not $\mu_{\omega', \beta'}$-semistable, its Harder-Narasimhan filtration would produce a subobject in $F_{\omega', \beta'}$ with positive slope. However, the slope $\mu_{\omega', \beta'}(\Phi(\cE)[-1]) = 0$ because $\Phi(\cE)$ has phase 1.

\subsubsection*{\underline{Step 3} (Giesker Stability)}
We now address part (i). Let $\cE$ be a twisted $\mu_{\omega', \beta'}$-semistable sheaf with Mukai vector $-\Phi(\mathbf{v})$, which has phase $0$ by construction. Suppose $\cF \subset \cE$ is a saturated subsheaf of the same slope. Then $\cF[1]$ is a subobject of $\cE[1]$ in $P_{\omega', \beta'}(1)$, where $P_{\omega', \beta'}$ denotes the slicing for $\tau' = \sigma_{\omega', \beta'}$. 

 As  $\tau'$ is  $\Phi(\mathbf{v})$-generic, $\mathbf{v}(\cF)$ must be proportional to $\mathbf{v}(\cE)$; otherwise $\cF[1]$ would define a new wall through $\tau'$. Since we work on a surface, the twisted Hilbert polynomial of $\cF$ is therefore proportional to that of $\cE$. This conclusion holds independently of $\beta'$, so we may disregard $\beta'$ and consider only $\omega'$-Gieseker stability.

For (ii), similarly, assume $\cE$ is $\omega'$-Gieseker stable with Mukai vector $-\Phi(\mathbf{v})$ and has a saturated subsheaf $\cF$ satisfying $\mu_{\omega', \beta'}(\cF) = \mu_{\omega', \beta'}(\cE)$. Therefore, $\cF$ and $\cE$ both have phase 1, making $\cF[1]$ a subobject of $\cE[1]$ in $P(\omega', \beta')(1)$. The same argument shows that $\cF$ and $\cE$ have identical twisted Hilbert polynomials, yielding a contradiction. The assertion follows since $\Phi \circ [-1]$ preserves S-equivalence classes.
\end{proof}

\bibliographystyle{plain}
\bibliography{main}

\begin{thebibliography}{10}

\bibitem{add16}
Nicolas Addington.
\newblock On two rationality conjectures for cubic fourfolds.
\newblock {\em Math. Res. Lett.}, 23(1):1--13, 2016.

\bibitem{BM14}
Arend Bayer and Emanuele Macr\`\i.
\newblock M{MP} for moduli of sheaves on {K}3s via wall-crossing: nef and movable cones, {L}agrangian fibrations.
\newblock {\em Invent. Math.}, 198(3):505--590, 2014.

\bibitem{BM}
Arend Bayer and Emanuele Macr\`i.
\newblock Projectivity and birational geometry of {B}ridgeland moduli spaces.
\newblock {\em J. Amer. Math. Soc.}, 27(3):707--752, 2014.

\bibitem{BV04}
Arnaud Beauville and Claire Voisin.
\newblock On the {C}how ring of a {$K3$} surface.
\newblock {\em J. Algebraic Geom.}, 13(3):417--426, 2004.

\bibitem{Bragg18}
Daniel Bragg and Max Lieblich.
\newblock Twistor spaces for supersingular {$K3$} surfaces.
\newblock {\em arXiv: Algebraic Geometry}, 2018.

\bibitem{Bridgeland08}
Tom Bridgeland.
\newblock Stability conditions on {$K3$} surfaces.
\newblock {\em Duke Math. J.}, 141(2):241--291, 2008.

\bibitem{B91}
Constantin B\u{a}nic\u{a}.
\newblock Smooth reflexive sheaves.
\newblock In {\em Proceedings of the {C}olloquium on {C}omplex {A}nalysis and the {S}ixth {R}omanian-{F}innish {S}eminar}, volume~36, pages 571--593, 1991.

\bibitem{Huy14}
D.~Huybrechts.
\newblock Curves and cycles on {K}3 surfaces.
\newblock {\em Algebr. Geom.}, 1(1):69--106, 2014.
\newblock With an appendix by C. Voisin.

\bibitem{Huy10}
Daniel Huybrechts.
\newblock Chow groups of {K}3 surfaces and spherical objects.
\newblock {\em J. Eur. Math. Soc. (JEMS)}, 12(6):1533--1551, 2010.

\bibitem{Huy17}
Daniel Huybrechts.
\newblock The {K}3 category of a cubic fourfold.
\newblock {\em Compos. Math.}, 153(3):586--620, 2017.

\bibitem{Huy}
Daniel Huybrechts and Manfred Lehn.
\newblock {\em The geometry of moduli spaces of sheaves}.
\newblock Aspects of Mathematics, E31. Friedr. Vieweg \& Sohn, Braunschweig, 1997.

\bibitem{HMS}
Daniel Huybrechts, Emanuele Macr\`i, and Paolo Stellari.
\newblock Stability conditions for generic {$K3$} categories.
\newblock {\em Compos. Math.}, 144(1):134--162, 2008.

\bibitem{HS}
Daniel Huybrechts and Paolo Stellari.
\newblock Proof of {C}\u ald\u araru's conjecture. {A}ppendix: ``{M}oduli spaces of twisted sheaves on a projective variety'' [in {\it {m}oduli spaces and arithmetic geometry}, 1--30, {M}ath. {S}oc. {J}apan, {T}okyo, 2006; mr2306170] by {K}. {Y}oshioka.
\newblock In {\em Moduli spaces and arithmetic geometry}, volume~45 of {\em Adv. Stud. Pure Math.}, pages 31--42. Math. Soc. Japan, Tokyo, 2006.

\bibitem{JL22}
Qingyuan Jiang and Naichung~Conan Leung.
\newblock Derived categories of projectivizations and flops.
\newblock {\em Adv. Math.}, 396:Paper No. 108169, 44, 2022.

\bibitem{KS16}
Bruno Kahn and Ramdorai Sujatha.
\newblock Birational motives {I}: {P}ure birational motives.
\newblock {\em Ann. K-Theory}, 1(4):379--440, 2016.

\bibitem{LYZ23}
Zhiyuan Li, Xun Yu, and Ruxuan Zhang.
\newblock Bloch's conjecture for (anti-)autoequivalences on {$K3$} surfaces, 2023.

\bibitem{LiZ22}
Zhiyuan Li and Ruxuan Zhang.
\newblock {Beauville–Voisin Filtrations on Zero-Cycles of Moduli Space of Stable Sheaves on K3 Surfaces}.
\newblock {\em International Mathematics Research Notices}, 06 2022.
\newblock rnac161.

\bibitem{Lie07}
Max Lieblich.
\newblock Moduli of twisted sheaves.
\newblock {\em Duke Math. J.}, 138(1):23--118, 2007.

\bibitem{LMS14}
Max Lieblich, Davesh Maulik, and Andrew Snowden.
\newblock Finiteness of {K}3 surfaces and the {T}ate conjecture.
\newblock {\em Ann. Sci. \'Ec. Norm. Sup\'er. (4)}, 47(2):285--308, 2014.

\bibitem{Lin18}
Hsueh-Yung Lin.
\newblock {Lagrangian Constant Cycle Subvarieties in Lagrangian Fibrations}.
\newblock {\em International Mathematics Research Notices}, 2020(1):14--24, 02 2018.

\bibitem{MZ20}
Alina Marian and Xiaolei Zhao.
\newblock On the group of zero-cycles of holomorphic symplectic varieties.
\newblock {\em \'{E}pijournal G\'{e}om. Alg\'{e}brique}, 4:Art. 3, 5, 2020.

\bibitem{OG13}
Kieran~G. O'Grady.
\newblock Moduli of sheaves and the {C}how group of {$K3$} surfaces.
\newblock {\em J. Math. Pures Appl. (9)}, 100(5):701--718, 2013.

\bibitem{SYZ20}
Junliang Shen, Qizheng Yin, and Xiaolei Zhao.
\newblock Derived categories of {$K3$} surfaces, {O}'{G}rady's filtration, and zero-cycles on holomorphic symplectic varieties.
\newblock {\em Compos. Math.}, 156(1):179--197, 2020.

\bibitem{vial22}
Charles Vial.
\newblock On the birational motive of hyper-{K}\"{a}hler varieties.
\newblock {\em J. Math. Pures Appl. (9)}, 163:577--624, 2022.

\bibitem{Vistoli89}
Angelo Vistoli.
\newblock Intersection theory on algebraic stacks and on their moduli spaces.
\newblock {\em Inventiones mathematicae}, 97:613--670, 1989.

\bibitem{voisin2003hodge}
Claire Voisin.
\newblock {\em Hodge Theory and Complex Algebraic Geometry II: Volume 2}, volume~77.
\newblock Cambridge University Press, 2003.

\bibitem{Voi08}
Claire Voisin.
\newblock On the {C}how ring of certain algebraic hyper-{K}\"{a}hler manifolds.
\newblock {\em Pure Appl. Math. Q.}, 4(3, Special Issue: In honor of Fedor Bogomolov. Part 2):613--649, 2008.

\bibitem{Voi15}
Claire Voisin.
\newblock Rational equivalence of 0-cycles on {$K3$} surfaces and conjectures of {H}uybrechts and {O}'{G}rady.
\newblock In {\em Recent advances in algebraic geometry}, volume 417 of {\em London Math. Soc. Lecture Note Ser.}, pages 422--436. Cambridge Univ. Press, Cambridge, 2015.

\bibitem{Yos01}
K\=ota Yoshioka.
\newblock Moduli spaces of stable sheaves on abelian surfaces.
\newblock {\em Math. Ann.}, 321(4):817--884, 2001.

\bibitem{yoshioka2006moduli}
K\=ota Yoshioka.
\newblock Moduli spaces of twisted sheaves on a projective variety.
\newblock In {\em Moduli spaces and arithmetic geometry}, volume~45 of {\em Adv. Stud. Pure Math.}, pages 1--30. Math. Soc. Japan, Tokyo, 2006.

\end{thebibliography}
\end{document}